\documentclass[11pt]{article}
\date{7  December, 2023} 
\usepackage{amsmath, latexsym, amsfonts, amssymb, amsthm, amscd}
\usepackage{mathtools} 
\usepackage[left=2cm, right=2cm, top=2.5cm, bottom=2.5cm]{geometry}
\usepackage{upquote} 
\usepackage{color}
\usepackage{setspace} 
\usepackage[labelfont=bf]{caption} 
\usepackage{stmaryrd} 
\usepackage{multicol} 
\usepackage[dvipsnames]{xcolor}
\usepackage[title]{appendix}

\usepackage[utf8]{inputenc}
\usepackage[T1]{fontenc}

\usepackage{kpfonts} 

\usepackage{dsfont} 

\usepackage[english]{babel} 

\usepackage{bm} 

\usepackage{enumitem} 

\usepackage{xcolor}
\definecolor{Maroon}{HTML}{ad2231}
\definecolor{webgreen}{HTML}{008000}
\usepackage{dsfont}
\usepackage{hyperref}
\hypersetup{colorlinks, breaklinks, urlcolor=Maroon, linkcolor=Maroon, citecolor=webgreen} 

\allowdisplaybreaks
\numberwithin{equation}{section}

\newtheorem{theorem}{Theorem}[section]
\newtheorem{lemma}[theorem]{Lemma}
\newtheorem{proposition}[theorem]{Proposition}

\newtheorem{condition}[theorem]{Condition}

\newtheorem{problem}[theorem]{Problem}

\theoremstyle{definition}

\newtheorem{example}[theorem]{Example}
\newtheorem{remark}[theorem]{Remark}

\usepackage{todonotes}

\usepackage{chngcntr}
\usepackage{apptools}
\AtAppendix{\counterwithin{lemma}{section}}

\newcommand{\refT}[1]{Theorem~\ref{#1}}

\newcommand{\refL}[1]{Lemma~\ref{#1}}

\newcommand{\refR}[1]{Remark~\ref{#1}}

\newcommand{\refS}[1]{Section~\ref{#1}}
\newcommand{\refSs}[1]{Sections~\ref{#1}}
\newcommand{\refSS}[1]{Section~\ref{#1}}
\newcommand{\refP}[1]{Problem~\ref{#1}}

\newcommand{\refE}[1]{Example~\ref{#1}}

\newcommand{\refApp}[1]{Appendix~\ref{#1}}
\newcommand{\refCond}[1]{Condition~\ref{#1}}

\newcommand\set[1]{\ensuremath{\{#1\}}}
\newcommand\bigset[1]{\ensuremath{\bigl\{#1\bigr\}}}

\newcommand\xpar[1]{(#1)}
\newcommand\bigpar[1]{\bigl(#1\bigr)}
\newcommand\Bigpar[1]{\Bigl(#1\Bigr)}

\newcommand\lrpar[1]{\left(#1\right)}
\newcommand\bigsqpar[1]{\bigl[#1\bigr]}
\newcommand\sqpar[1]{[#1]}

\newcommand\lrsqpar[1]{\left[#1\right]}

\newcommand\E{\operatorname{\mathbb E{}}}
\renewcommand\P{\operatorname{\mathbb P{}}}

\newcommand\Var{\operatorname{Var}}
\newcommand\Cov{\operatorname{Cov}}
\newcommand\Ge{\operatorname{Ge}}
\newcommand\Po{\operatorname{Po}}

\newcommand\gd{\delta}
\newcommand\gD{\Delta}

\newcommand\gG{\Gamma}
\newcommand\kk{\kappa}
\newcommand\gl{\lambda}

\newcommand\gs{\sigma}

\newcommand\gss{\sigma^2}

\newcommand\gth{\theta}

\newcommand\gU{\Upsilon}

\newcommand\cD{\mathcal D}
\newcommand\cE{\mathcal E}

\newcommand\cL{{\mathcal L}}

\newcommand\cP{\mathcal P}

\newcommand\cT{{\mathcal T}}

\newcommand\tS{{\widetilde S}}

\newcommand\bW{\mathbf W}

\newcommand\qq{^{1/2}}
\newcommand\qqw{^{-1/2}}

\newcommand\pfitemx[1]{\par#1:}
\newcommand\pfitemref[1]{\pfitemx{\ref{#1}}}
\newcommand\ntoo{\ensuremath{{n\to\infty}}}

\newcommand\kktoo{\ensuremath{{\kk\to\infty}}}
\newcommand\indic[1]{\mathbf{1}_{#1}}

\newcommand{\tend}{\longrightarrow}
\newcommand\dto{\overset{\mathrm{d}}{\tend}}
\newcommand\pto{\overset{\mathrm{p}}{\tend}}

\newcommand\eqd{\overset{\mathrm{d}}{=}}

\newcommand\bn{{\mathbf{n}}}
\newcommand\bp{{\mathbf{p}}}
\newcommand\bnk{{\bn_\kk}}

\newcommand\bpn{{\bp(\bn)}}
\newcommand\bpnk{{\bp(\bnk)}}
\newcommand\bbT{\mathbb T}
\newcommand\bbB{\mathbb{B}}
\newcommand\bbE{\mathbb{E}}
\newcommand\bbR{\mathbb R}

\newcommand\bbNo{\mathbb{N}_0}
\newcommand\Wbn{W_{\bn}^{\mathrm{b}}}
\newcommand\Tnk{\cT_\bnk}
\newcommand\Ttn{\widetilde{\cT}}
\newcommand\cTn{\cT_{\bn}}
\newcommand\mujk{\mu_{j\kk}}

\newcommand\pip{\pi_{\bp}}
\newcommand\pipn{\pi_{\bpn}}

\newcommand\pipnk{\pi_{\bpnk}}
\newcommand\etap{\eta_{\bp}}
\newcommand\gammap{\gamma_{\bp}}
\newcommand\tr{\widetilde r}
\newcommand\ENTX{|\bn|\pi_\bpn(T)}
\newcommand\muk{\mu_{\bnk}}
\newcommand\hGamma{\widehat{\gG}}
\newcommand\hgamma{\widehat{\gamma}}
\newcommand\xGamma{\widetilde{\gG}}
\newcommand\xgamma{\widetilde{\gamma}}
\newcommand\heta{\widehat{\eta}}
\newcommand\N{\mathrm{N}}
\newcommand\gdx[2]{\gd_{#1#2}}
\newcommand\ppi{\cP_1(\bbNo)}
\newcommand\bbTlab{\mathbb{T}^{\mathrm{lab}}}
\newcommand\cTlab{\mathcal{T}^{\mathrm{lab}}}
\newcommand\ocTlabd{\overline{\mathcal{T}_\bd}^{\mathrm{lab}}}
\newcommand\un{^{\mathrm{un}}}
\newcommand\bbTun{\mathbb{T}\un}
\newcommand\piun{\pi\un}
\newcommand\cTun{\mathcal{T}\un}
\newcommand\bd{{\mathbf{d}}}
\newcommand\bdk{{\bd_\kk}}
\newcommand\bbD{\mathbb{D}}
\newcommand\Ord{\operatorname{Ord}}
\newcommand\ord[1]{|\Ord(#1)|}
\newcommand\Tord{\overline{T}}
\newcommand\cTp{\cT_\bp}
\newcommand\ran{\fr}
\newcommand\fr{{\mathrm{fr}}}
\newcommand\cTran{\cT^\ran}
\newcommand\cTlabran{\cT^{\mathrm{lab},\ran}}
\newcommand\bT{\overline{T}}
\newcommand\hf{\widehat{f}}
\newcommand\hF{\widehat{F}}
\newcommand\bbTK{\bbT^{(K)}}
\newcommand\hw{\widehat{w}}
\newcommand\bw{{\mathbf{w}}}
\newcommand\hbw{\widehat{\bw}}
\newcommand\xS{\widetilde{\mathbb{S}}}
\newcommand\nx{n_\kk}
\newcommand\dgamma{\gamma^*}
\newcommand\dGamma{\Gamma^*}
\newcommand\Unif{\operatorname{Unif}}

\begingroup
  \divide\count255 by 60
  \multiply\count255 by -60
  \advance\count255 by \time
  \ifnum \count255 < 10 \xdef\klockan{\the\count1.0\the\count255}
  \else\xdef\klockan{\the\count1.\the\count255}\fi
\endgroup

\begin{document}
\title{Fringe trees for random trees with given vertex degrees}
\author{Gabriel Berzunza Ojeda\footnote{ {\sc Department of Mathematical Sciences, University of Liverpool, United Kingdom.} E-mail: gabriel.berzunza-ojeda@liverpool.ac.uk URL: \href{https://www.liverpool.ac.uk/mathematical-sciences/staff/gabriel-berzunza-ojeda/}{https://www.liverpool.ac.uk/mathematical-sciences/staff/gabriel-berzunza-ojeda/} }, \,\, Cecilia Holmgren\footnote{ {\sc Department of Mathematics, Uppsala University, Sweden.} E-mail: cecilia.holmgren@math.uu.se \href{https://katalog.uu.se/empinfo/?id=N5-824}{https://katalog.uu.se/empinfo/?id=N5-824}} \, \, and \, \, Svante Janson\footnote{ {\sc Department of Mathematics, Uppsala University, Sweden.} E-mail: svante.janson@math.uu.se URL: \href{http://www2.math.uu.se/~svante/}{http://www2.math.uu.se/~svante/}} 
}

\newcommand\extrafootertext[1]{%
    \bgroup
    \renewcommand\thefootnote{\fnsymbol{footnote}}%
    \renewcommand\thempfootnote{\fnsymbol{mpfootnote}}%
    \footnotetext[0]{#1}%
    \egroup
}
\extrafootertext{%
Supported by the Knut and Alice Wallenberg Foundation; 
 Ragnar Söderberg's Foundation; 
the Swedish Research Council.  
}

\maketitle


\vspace{0.1in}

\begin{abstract} 
We prove asymptotic normality for the number of fringe subtrees isomorphic
to any given tree in uniformly random trees with given vertex degrees. As
applications, we also prove corresponding results for random labelled trees
with given vertex degrees, for random simply generated trees (or conditioned
Galton--Watson trees), and for additive functionals. 

The key tool for our work is an extension to the multivariate setting of a
theorem by Gao and Wormald (2004), which provides a way to show asymptotic
normality by analysing the behaviour of sufficiently high factorial moments. 
\end{abstract}



\section{Introduction and main results}\label{S:Intro}

In this paper, we consider fringe trees of 
three types of random trees.
In the main parts of the paper, we consider 
random plane trees with given vertex statistics, i.e., a given number of
vertices of each degree.
As applications of these results, we also give corresponding results
for random labelled trees with given vertex degrees,
and for random simply generated trees (or conditioned Galton--Watson trees).
The main results are laws of large numbers and central limit theorems for
the number of fringe trees of a given type.

Let $\mathbb{T}$ be the set of all (finite) plane rooted trees (also called
ordered rooted trees); see e.g., \cite{Drmota2009}. Denote the size, i.e.\ the
number of vertices, of a tree $T$ by $|T|$. The (out)degree of a vertex $v
\in T$, denoted $d_{T}(v)$, is its number of children in $T$; 
thus leaves have degree $0$ and all other vertices have strictly positive
degree. 
The \emph{degree statistic} of a rooted tree $T$ is the sequence $\mathbf{n}_{T} =
(n_{T}(i))_{i \geq 0}$, where $n_{T}(i) := |\{v \in T: d_{T}(v) = i\}|$ is
the number of vertices of $T$ with $i$ children.
We have
\begin{eqnarray} \label{eq30}
|T| = \sum_{i \geq 0} n_{T}(i) = 1 + \sum_{i \geq 0} in_{T}(i).
\end{eqnarray}
\noindent A sequence $\mathbf{n} = (n(i))_{i \geq 0}$ is the degree
statistic of some tree if and only if 
$\sum_{i \geq 0} n(i) = 1 + \sum_{i \geq 0} in(i)$. 
For such sequences, we 
let $|\mathbf{n}| \coloneqq \sum_{i \geq 0} n(i)$ be the size of 
$\mathbf{n}$, and we
write $\mathbb{T}_{\mathbf{n}}$ for
the set of plane rooted trees with degree statistic $\mathbf{n}$.
We let $\mathcal{T}_{\mathbf{n}}$ be a uniformly random element of the set
$\mathbb{T}_{\mathbf{n}}$, and  we denote this by  
$\mathcal{T}_{\mathbf{n}} \sim {\rm Unif}(\mathbb{T}_{\mathbf{n}})$.

For $T \in \mathbb{T}$ and a vertex $v \in T$, let $T_{v}$ be the subtree of
$T$ rooted at $v$
consisting of $v$ and all its
descendants. We call $T_{v}$ a fringe (sub)tree of $T$.
We regard $T_v$ as an element of $\bbT$ and
let, for $T, T^{\prime} \in \mathbb{T}$, 
\begin{eqnarray} \label{eq2}
N_{T^{\prime}}(T) \coloneqq  |\{v \in T: T_{v} = T^{\prime}\}| = \sum_{v \in T} \mathbf{1}_{\{T_{v} = T^{\prime}\}},
\end{eqnarray}
\noindent i.e., the number of fringe subtrees of $T$ that are equal (i.e., isomorphic to) to $T^{\prime}$. A random fringe subtree $T^{\ran}$ of $T \in \mathbb{T}$ is the random rooted tree obtained by taking the fringe subtree $T_{v}$ at a uniform random vertex $v \in T$. Thus, the distribution of $T^{\ran}$ is given by
\begin{align}\label{sw1}
\mathbb{P}(T^{\ran} = T^{\prime})=\frac{N_{T'}(T)}{|T|}, \hspace*{4mm} \text{for} \hspace*{2mm} T^{\prime} \in \mathbb{T}. 
\end{align}

We prove an asymptotic result on the distribution of a random fringe subtree
in a random rooted plane tree with a given degree statistic. In order to
state the theorem, we need a little more terminology. 
(See also \refSS{SSnot} for some notation.)
For a degree statistic $\mathbf{n}$, denote by $\mathbf{p}(\mathbf{n}) = (p_{i}(\mathbf{n}))_{i \geq 0}$ its (empirical) degree distribution, i.e., 
\begin{align}\label{pin}
p_{i}(\mathbf{n}) := \frac{n(i)}{|\mathbf{n}|}, \hspace*{4mm} \text{for} \hspace*{2mm} i \geq 0. 
\end{align}
In this paper, we assume for convenience the following condition. 
\begin{condition} \label{Condition1}
$\mathbf{n}_{\kappa} = (n_{\kappa}(i))_{i \geq 0}$, $\kk\ge1$, 
are degree
statistics such that 
as $\kappa \rightarrow \infty$:
\begin{enumerate}[label=\upshape(\roman*)]
\item\label{Cond1a} 
$|\mathbf{n}_{\kappa}| \rightarrow \infty$, 
\item \label{Cond1b}
For every $i \geq 0$, we have
$p_{i}(\mathbf{n}_{\kappa}) \rightarrow p_{i}$, 
where $\mathbf{p} = (p_{i})_{i \geq 0}$ is a 
probability distribution on $\bbNo$.
\end{enumerate} 
\end{condition}

\begin{remark}
The condition that $\bp$ is a probability distribution is no restriction.
In fact, the degree distribution $\bpnk$ has mean 
\begin{align}\label{mg1}
\sum_{i \geq 0}  ip_i(\bnk)
=\frac{1}{|\bnk|}\sum_{i \geq 0}  i n_\kk(i)
=\frac{|\bnk|-1}{|\bnk|}
<1,  
\end{align}
and thus the sequence of
distributions $\bpnk$ is always tight. Hence, if $p_i(\bnk)\to p_i$, for every
$i\ge0$, then $\bp=(p_i)_{i \geq 0}$ is a probability distribution. 
Note also that \ref{Cond1b}  says that
$\mathbf{p}(\mathbf{n}_{\kappa})$ converges weakly 
to $\mathbf{p}$, as
$\kappa \rightarrow \infty$. 
(As is well known, this is equivalent to convergence in total variation.)
\end{remark}

By \eqref{mg1} and Fatou's lemma, if \refCond{Condition1} holds, then
$\sum_{i \geq 0} ip_i\le1$. Conversely, it is easily seen that any such
probability distribution $\bp$ is the limit of $\bpnk$ for some sequence of
degree statistics $\bnk$.
In other words, the set of probability distributions $\bp$ that can appear as
limits in Condition \ref{Condition1} is precisely
the set of probability distributions $\bp$ on $\bbNo$ with mean 
$\sum_{i \geq 0} ip_i\le1$; we denote this set by $\ppi$.

For a probability distribution  $\mathbf{p} = (p_{i})_{i \geq  0}\in\ppi$,
let $\cTp$ be a Galton--Watson tree with offspring distribution $\bp$, and
define 
$\pip$ as the distribution of $\cTp$, i.e.,
(with $0^0:=1$ as usual)
\begin{align}\label{pip}
\pi_{\mathbf{p}}(T) 
\coloneqq 
\P\xpar{\cTp=T}
=\prod_{i \geq 0} p_{i}^{n_{T}(i)}
=\prod_{i \in\cD(T)} p_{i}^{n_{T}(i)}
, \qquad \text{for} \, \, \, T \in \mathbb{T},
\end{align}
where 
\begin{align}\label{cD}
\cD(T):=\set{i:n_T(i)>0}=\set{d_T(v):v\in T},   
\end{align}
the set of degrees that
appear in $T$.
Note that $\pi_\bp(T)=0 \iff p_i=0 \text{ for some }i\in\cD(T)$.
In particular, if $\bnk$ and $\bp$ are as in \refCond{Condition1}, then
$\pip(T)=0$ if and only if $n_{\kappa}(i)=o(|\bnk|)$ for some $i\in\cD(T)$.

We first give a law of large numbers for the number
of fringe trees of a given type
in a random rooted plane tree with a given degree
statistic. 
The proofs of this and the following theorems are given in later sections.

\begin{theorem} \label{TheoLLN}
Let $\mathbf{n}_{\kappa}$, $\kk\ge1$, 
be some degree statistics that satisfy Condition
\ref{Condition1}, and let 
$\mathcal{T}_{\mathbf{n}_{\kappa}} \sim {\rm Unif}(\mathbb{T}_{\mathbf{n}_{\kappa}})$. 
For every fixed $T \in \mathbb{T}$,
as \kktoo:
\begin{enumerate}[label=\upshape(\roman*)]
\item (Annealed version)\quad 
$\displaystyle \mathbb{P}(\mathcal{T}^{\ran}_{\mathbf{n}_{\kappa}} = T) 
=
\frac{\E[N_{T}(\mathcal{T}_{\mathbf{n}_{\kappa}})]}{|\mathbf{n}_{\kappa}|}
\rightarrow \pi_{\mathbf{p}}(T)$. 
\label{Annealed}

\item (Quenched version)\quad 
$\displaystyle \mathbb{P}(\mathcal{T}^{\ran}_{\mathbf{n}_{\kappa}} 
= T\mid \Tnk) =
  \frac{N_{T}(\mathcal{T}_{\mathbf{n}_{\kappa}})}{|\mathbf{n}_{\kappa}|}
  \rightarrow \pi_{\mathbf{p}}(T)$ 
in probability. 
\label{Quenched}
\end{enumerate}
In other words,
the random fringe tree converges in distribution as $\kk\to\infty$:
\ref{Annealed} says 
$\cTran_{\bnk}\dto \cTp$,
or equivalently $\cL(\cTran_{\bnk})\to \cL(\cTp)$, 
and \ref{Quenched} is the conditional version
$\cL\bigpar{\cTran_{\bnk}\mid \Tnk}\pto \cL(\cTp)$. 
\end{theorem}

\begin{remark}
Similar results are known for several other models of random trees. 
In particular, a version of Theorem \ref{TheoLLN} was proved by Aldous
\cite{Aldous1991} for conditioned Galton--Watson trees with finite offspring
variance; this was extended to general simply generated trees in
\cite[Theorem 7.12]{Svante2012}. 
In those cases, the degree statistic is random,
but  Condition \ref{Condition1} holds in probability,
with a non-random  limiting probability distribution $\mathbf{p}$. 
We return to simply generated trees in  \refSS{SimplyTrees}.
Another standard example is family trees of Crump--Mode--Jagers branching
processes (which includes e.g.\ random recursive trees, binary search trees
and preferential attachment trees);
see e.g.\
\cite{Aldous1991} and \cite[Theorem 5.14]{SJ306}.
\end{remark}

Theorem \ref{TheoLLN} is thus a law of large numbers for the number
of fringe trees of a given type. 
In this work, we also study the fluctuations and prove a central limit theorem
for this number;
we furthermore show that this holds jointly for different types of fringe trees.

For a probability distribution $\mathbf{p} = (p_{i})_{i \geq 0} \in\ppi$ and $T, T^{\prime} \in \mathbb{T}$, let 
\begin{align} \label{eq12}
\eta_{\mathbf{p}}(T,T^{\prime}) \coloneqq (|T|-1)(|T^{\prime}|-1) - \sum_{i\geq 0}\frac{n_{T}(i)n_{T^{\prime}}(i)}{p_{i}} 
,\end{align}
where we interpret $0/0:=0$,
and
\begin{align} 
\gamma_{\mathbf{p}}(T, T) &  := \pi_{\mathbf{p}}(T) + \eta_{\mathbf{p}}(T,T)(\pi_{\mathbf{p}}(T))^{2}, \label{eq20}\\
\gamma_{\mathbf{p}}(T, T') & 
:= N_{T'}(T)  \pi_{\mathbf{p}}(T) +  N_{T}(T') \pi_{\mathbf{p}}(T') 
+\eta_{\mathbf{p}}(T,T') \pi_{\mathbf{p}}(T) \pi_{\mathbf{p}}(T'), 
\qquad T\neq T'. 
\label{eq21}
\end{align}
Note that $\etap(T,T')=-\infty$ if $p_i=0$ for some $i\in\cD(T)\cap\cD(T')$.
In this case, $\pip(T)=\pip(T')=0$, and we interpret $\infty\cdot0:=0$
in \eqref{eq20}--\eqref{eq21}; thus $\gamma_{\mathbf{p}}(T,T')$ is always finite.

\begin{theorem} \label{TheoCLT}
Let $\mathbf{n}_{\kappa}$, $\kk\ge1$, be some degree statistics that satisfy
Condition  \ref{Condition1} and let 
$\mathcal{T}_{\mathbf{n}_{\kappa}} \sim {\rm Unif}(\mathbb{T}_{\mathbf{n}_{\kappa}})$. 
For a fixed $m \geq 1$, let
$T_{1}, \dots, T_{m} \in \mathbb{T}$ be a fixed sequence of 
rooted plane trees.
Then, as $\kappa \rightarrow \infty$,
\begin{align} 
\E N_{T_{i}}(\mathcal{T}_{\mathbf{n}_{\kappa}})  & 
=  \pi_{\mathbf{p}}(T_{i})|\mathbf{n}_{\kappa}|  + o(|\mathbf{n}_{\kappa}|), \label{eq6}\\
{\rm Var}(N_{T_{i}}(\mathcal{T}_{\mathbf{n}_{\kappa}})) & 
=   \gamma_\bp(T_i,T_{i})|\mathbf{n}_{\kappa}| + o(|\mathbf{n}_{\kappa}|), \label{eq15} \\
{\rm Cov}\bigpar{N_{T_{i}}(\mathcal{T}_{\mathbf{n}_{\kappa}}),
  N_{T_{j}}(\mathcal{T}_{\mathbf{n}_{\kappa}})} & 
=   \gamma_\bp(T_i,T_j)|\mathbf{n}_{\kappa}| + o(|\mathbf{n}_{\kappa}|), 
\label{eq13}
\end{align}
for $1 \leq i,j \leq m$,
and
\begin{align} \label{eq7}
\left(\frac{ N_{T_{1}}(\mathcal{T}_{\mathbf{n}_{\kappa}}) -
  \E[N_{T_{1}}(\mathcal{T}_{\mathbf{n}_{\kappa}})]}{\sqrt{|\mathbf{n}_{\kappa}|}}, 
\dots, 
\frac{ N_{T_{m}}(\mathcal{T}_{\mathbf{n}_{\kappa}}) -
  \E[N_{T_{m}}(\mathcal{T}_{\mathbf{n}_{\kappa}})]}{\sqrt{|\mathbf{n}_{\kappa}|}}
  \right) 
\dto\N(0, \Gamma_{\mathbf{p}}),
\end{align}
\noindent where the covariance matrix 
$\Gamma_{\mathbf{p}} :=(\gamma_{\mathbf{p}}(T_{i}, T_{j}))_{i,j=1}^m$. 
Furthermore, in \eqref{eq7}, 
we can replace
$\E[N_{T_{i}}(\mathcal{T}_{\mathbf{n}_{\kappa}})]$ 
by $|\mathbf{n}_{\kappa}| \pi_{\mathbf{p}(\mathbf{n}_{\kappa})}(T_{i})$. 

If\/ $T\in\bbT$ with $\pip(T)>0$ and $|T|>1$, then $\gammap(T,T)>0$ and thus 
\eqref{eq15} and \eqref{eq7} (with $m=1$) show 
that $N_T(\Tnk)$ is asymptotically normal, with
\begin{align}\label{eq70}
\frac{ N_{T}(\mathcal{T}_{\mathbf{n}_{\kappa}}) -
  \E[N_{T}(\mathcal{T}_{\mathbf{n}_{\kappa}})]}
{\sqrt{\Var (N_T(\Tnk))}}
\dto \N(0,1), \qquad \kappa \rightarrow \infty.
\end{align}
\end{theorem}

The case $|T|=1$ is trivial, with $N_T(\Tnk)=n_\kk(0)$ non-random.
Ignoring this case, \refT{TheoCLT} shows that $N_T(\Tnk)$ is asymptotically
normal when $\pip(T)>0$.
On the other hand, if $\pi_\bp(T)=0$, then also $\gamma_\bp(T,T)=0$, and the
theorems above
do not give precise information on the asymptotic distribution of
$N_T(\Tnk)$. 
In this case, the following theorems are more precise.

\begin{theorem}\label{TA}
  Let $T\in\bbT$ be a fixed tree.
Then, uniformly for all degree statistics $\bn=(n(i))_{i \geq 0}$,
\begin{align}\label{ta1}
  \E N_T(\cTn)&= |\bn| \pi_{\bpn}(T) + O(1),
\\\label{ta10}
  \Var N_T(\cTn)&= |\bn| \gamma_{\bpn}(T,T) + O(1).
\end{align}
More generally, if $T, T^{\prime}\in\bbT$, then
\begin{align}\label{ta20}
  \Cov\bigpar{N_{T}(\cTn),N_{T^{\prime}}(\cTn)}
= |\bn|\gamma_\bpn(T,T^{\prime})+O(1).
\end{align}
\end{theorem}

In view of \eqref{ta1}, we define, for any degree statistic $\bn$ and tree
$T\in\bbT$, 
\begin{align}\label{mu}
  \mu_\bn(T):=|\bn|\pipn(T)
=|\bn|\prod_{i\ge0} p_i(\bn)^{n_T(i)}
=|\bn|\prod_{i\in\cD(T)} p_i(\bn)^{n_T(i)}
.\end{align}
This is thus a convenient approximation of $\E N_T(\cTn)$.
We define also
\begin{align} \label{gw1}
\heta_{\mathbf{p}}(T,T^{\prime}) \coloneqq 
\xpar{\pip(T)\pip(T')}\qq\eta_\bp(T,T'),
\qquad \text{if } \, \pip(T),\pip(T')>0,
\end{align}
and extend this by continuity to the case $\pip(T)\pip(T')=0$;
this yields by \eqref{eq12} the general formula
\begin{align} \label{gw2}
\heta_{\mathbf{p}}(T,T^{\prime}) =
\xpar{\pip(T)\pip(T')}\qq(|T|-1)(|T^{\prime}|-1) - 
\sum_{i\geq 0}{n_{T}(i)n_{T^{\prime}}(i)}
\prod_{j \in\cD(T)\cup\cD(T')} p_j^{(n_T(j)+n_{T'}(j))/2-\gdx{i}{j}}
.\end{align}
We interpret again $0\cdot\infty:=0$; thus the sum in \eqref{gw2} is finite
also if  $p_i=0$ for some $i \in\cD(T)\cup\cD(T')$. 
In fact, \eqref{gw2} is a polynomial in $p_0\qq,p_1\qq,\dots$, and is thus
continuous in $\bp$ as asserted.

Similarly, we define
\begin{align} \label{gw3}
\hgamma_{\mathbf{p}}(T,T^{\prime}) \coloneqq 
\xpar{\pip(T)\pip(T')}\qqw\gamma_\bp(T,T'),
\qquad \text{if } \, \pip(T),\pip(T')>0,
\end{align}
and extend this by continuity, which 
by \eqref{eq20}--\eqref{eq21} and \eqref{gw1} yields
\begin{align} 
\hgamma_{\mathbf{p}}(T, T) &  := 1 + \heta_{\mathbf{p}}(T,T), \label{gw20}\\
\hgamma_{\mathbf{p}}(T, T') & 
:= N_{T'}(T) \prod_{i \in\cD(T)}p_i^{(n_i(T)-n_i(T'))/2} 
+  N_{T}(T') \prod_{i \in\cD(T')}p_i^{(n_i(T')-n_i(T))/2}
+\heta_{\mathbf{p}}(T,T'),
\qquad T\neq T'. 
\label{gw21}
\end{align}
Note that $N_{T'}(T)>0$ implies 
$\cD(T')\subseteq\cD(T)$ and
$n_T(i)\ge n_{T'}(i)$, for $i \geq 0$; thus \eqref{gw21} 
always yields a finite value (again interpreting $0\cdot\infty:=0$);
again, this is a polynomial in $p_0\qq,p_1\qq,\dots$, and thus
continuous in $\bp$.

\begin{theorem} \label{TheoCLT2}
Let $\mathbf{n}_{\kappa}$, $\kk\ge1$, be some degree statistics that satisfy
Condition  \ref{Condition1} and let 
$\mathcal{T}_{\mathbf{n}_{\kappa}} \sim {\rm Unif}(\mathbb{T}_{\mathbf{n}_{\kappa}})$. 
For fixed $m \geq 1$, let
$T_{1}, \dots, T_{m} \in \mathbb{T}$ be a fixed sequence of 
rooted plane trees such that, as $\kk\to\infty$,
$\muk(T_i):=|\bnk|\pi_\bpnk(T_i)\to\infty$ for each $1 \le i\le m$.
Then, 
\begin{align} \label{eq77}
\lrpar{\frac{ N_{T_{1}}(\Tnk) - \E[N_{T_{1}}(\Tnk)]}{\sqrt{\muk(T_1)}},
\dots, 
\frac{ N_{T_{m}}(\Tnk) - \E[N_{T_{m}}(\Tnk)]}{\sqrt{\muk(T_m)}}}
\dto
 {\rm N}(0, \hGamma_{\mathbf{p}}),  \hspace*{4mm} \text{as} \hspace*{2mm}
  \kappa \rightarrow \infty, 
\end{align}
\noindent where the covariance matrix 
$\hGamma_{\mathbf{p}} :=(\hgamma_{\mathbf{p}}(T_{i}, T_{j}))_{i,j=1}^m$. 
Furthermore, in \eqref{eq77}, 
we can replace
$\mathbb{E}[N_{T_{i}}(\mathcal{T}_{\mathbf{n}_{\kappa}})]$ 
by $\muk(T_i)$. 

Moreover, $\hgamma_\bp(T,T)>0$, and thus the asymptotic normality
\eqref{eq70} holds,
except in the following three exceptional cases:
\begin{enumerate}[label=\upshape(\roman*)]
  \item\label{exc1} $|T|=1$,
  \item\label{exc2}
$T$ is a path and $p_1=1$,
\item\label{exc3} 
$T$ is a star 
with a root of degree $d$ joined to $d$ leaves, 
and $p_0=1$.
  \end{enumerate}
\end{theorem}

The exceptional cases \ref{exc2} and \ref{exc3} 
are discussed further in \refE{Eexc}.

\begin{remark}
\refT{TheoCLT2} shows that
excluding the exceptional cases \ref{exc1}--\ref{exc3},
the condition $\mu_\bnk(T)\to\infty$, as $\kappa \rightarrow \infty$,  
is sufficient for asymptotic normality of $N_T(\Tnk)$.
This condition is also necessary, since otherwise (at least for a subsequence)
$\E N_T(\Tnk)=O(1)$ by \eqref{ta1}, and 
since $N_T(\Tnk)$ is integer-valued, 
it is easy to see that then it cannot converge to a non-degenerate normal
distribution for any normalization.
\end{remark}

\begin{problem}\label{ProbGamma}
In \refT{TheoCLT}, suppose 
that $T_1,\dots,T_m$ are distinct
with
$|T_i|>1$ and $\pip(T_i)>0$ for every $1 \leq i \leq m$. 
\refT{TheoCLT} says that $\gammap(T_i,T_i)>0$, for every $1 \leq i \leq m$.
Is the covariance matrix $\Gamma_\bp$ non-singular?
\end{problem}

In the case of critical conditioned Galton--Watson trees with finite offspring
variance,  
(joint) normal convergence of the subtree counts in analogy to \eqref{eq7}
was proved in  \cite[Corollary 1.8]{Svante2016} (together with convergence
of mean and variance).
Indeed,
\cite[Theorem 1.5]{Svante2016}  proved, more generally, asymptotic normality of
additive functionals that are defined via toll functions (under some
conditions); see \refSS{AdditiveFunct} for further discussion on additive functionals. 

\begin{remark}
Results on asymptotic normality for fringe tree counts have also been proved
earlier for several other classes of random trees.
For example, for binary search trees see
\cite{Devroye1}, 
\cite{Devroye2},
\cite{Chang},
\cite{FlajoletGM1997},
\cite{SJ296};
for random recursive trees see
\cite{FengMahmoud},
\cite{SJ296};
for increasing trees see \cite{Fuchs2012};
for $m$-ary search trees and preferential attachment trees see
\cite{SJ309};
for random tries see
\cite{SJ347}.
\end{remark}

Our approach relies on 
a multivariate version of the 
Gao--Wormald theorem \cite[Theorem 1]{Gao2004}; see Theorem \ref{TGW} in
Appendix \ref{AppendixA}.
The original  Gao--Wormald theorem \cite{Gao2004} 
provides a way to show asymptotic normality by analysing the behaviour of
sufficiently high factorial moments. 
(Typically, factorial moments are more convenient than standard moments
in combinatorics.) 
Our multivariate version \ref{TGW} extends this by
considering joint factorial moments.
In our
framework, this is very convenient since we can precisely compute the joint
factorial moments of the subtree counts in \eqref{eq2} for random trees
with given degree statistics.
(Another, closely related, multivariate version of the 
Gao--Wormald theorem has independently been shown recently by
Hitczenko and Wormald \cite{HitczenkoWormald};
see further \refApp{AppendixA}.)

The (one dimensional) Gao--Wormald theorem has been used before by Cai
and Devroye \cite{Cai2016} to study large fringe trees in critical
conditioned Galton--Watson trees with finite offspring variance. Indeed,
they considered fringe subtree counts of a sequence of trees instead of a
fixed tree. In particular, they showed that asymptotic normality still holds
in some regimes, while in others there is a Poisson limit. In a
forthcoming work, we will study the case of not fixed fringe trees in the
framework of random
trees with given degrees.



\subsection{Organization of the paper}
Some standard facts on coding trees by walks are recalled in \refS{Strees};
these facts are used in \refS{Smoments} to give 
exact formulas for factorial moments of $N_T(\cTn)$.
These formulas are then used in \refSs{SpfLLN+A}--\ref{SpfCLT}
to prove our main results. 

Applications to labelled trees with given vertex degrees, simply generated trees and additive functionals are given in 
\refSs{Slabelled}--\ref{AdditiveFunct}.

\refApp{AppendixA} contains a general statement and proof of the
multivariate version of the Gao--Wormald theorem that we use in our proof of
the main theorems.
\refApp{AppendixB} uses that theorem to give a new simple proof 
(in two cases)
of a known result on 
asymptotic normality of degree statistics in conditioned Galton--Watson trees
(\refT{TheoDegreeS}) that we use in the proof in \refS{SimplyTrees}. 

\subsection{Some notation}\label{SSnot}
In addition to the notation introduced above, we use the following standard
notation. 

We let $\mathbb{Z} \coloneqq \{\dots, -1, 0, 1, \dots
\}$, $\mathbb{N} \coloneqq \{1, 2, \dots \}$, $\mathbb{N}_{0} \coloneqq \{0,
1, 2, \dots \}$. 
We let $0$ denote also vectors and matrices with all elements $0$
(the dimension will be clear from the context).
We use standard $o$ and $O$ notation, for sequences and
functions of a real variable. Recall that $a_\kk=\Theta(b_\kk)$
means $a_\kk=O(b_\kk)$ and $b_\kk=O(a_\kk)$.

For two sequence of positive real numbers $(a_{n})_{n \geq 1}$ and $(b_{n})_{n \geq 1}$, we write $a_{n} \ll b_{n}$ or $b_{n} \gg a_{n}$ if and only if $a_{n}/b_{n} \rightarrow 0$ as $n \rightarrow \infty$.

$\indic{\cE}$ is the indicator function of an event $\cE$, and
$\gd_{ij}:=\indic{\set{i=j}}$ is Kronecker's delta.

For $x \in \mathbb{R}$ and $q \in \mathbb{N}_{0}$, we let $(x)_{q} \coloneqq
x(x-1)\cdots(x-q+1)$ denote the $q$th falling factorial of $x$. 
(Here $(x)_{0} \coloneqq 1$. 
Note that $(x)_{q} = 0$ whenever $x \in \mathbb{N}_{0}$ and $x-q+1\leq0$.)

We interpret $0/0=0$ and $0\cdot\infty=0$.

We use $\dto$ for convergence in distribution, 
and $\pto$ for convergence in probability, 
for a sequence of random variables in some metric space. 
Also, $\cL(X)$ denotes the distribution of $X$, and
$\stackrel{\rm d}{=}$ means equal in distribution. 
We write ${\rm N}(0,\Gamma)$ for the multivariate normal distribution
with mean vector $0$ and covariance matrix $\Gamma \coloneqq
(\gamma_{ij})_{i,j=1}^m$, for $m \in \mathbb{N}$. 
(This includes the case $\Gamma = 0$; in this case
$X\sim{\rm N}(0,\Gamma)$ means that $X=0\in\bbR^m$ a.s.)

Unspecified limits are as $\kktoo$.

\section{Trees and walks}\label{Strees}

For $k \in \mathbb{N}$, we view a sequence $x = (x(0), x(1), \dots, x(k)) \in \mathbb{Z}^{k+1}$ as a walk with steps (or increments) given by $\Delta x(i) \coloneqq x(i)-x(i-1)$, for $1 \leq i \leq k$. Define the set of all (discrete) bridges finishing in position $-1$ at time $k$, as
\begin{eqnarray}\label{tw1}
\mathbb{B}^{k} := \bigset{ (x(0), x(1), \dots, x(k)) \in \mathbb{Z}^{k+1}: \, x(0) = 0, \, \Delta x(i) \geq -1, \, \text{for} \, 1 \leq i \leq k, \, \text{and} \, x(k)=-1 }. 
\end{eqnarray}

\noindent For $1 \leq i \leq k$ and $x = (x(0), x(1), \dots, x(k)) \in
\mathbb{B}^{k}$, define $\omega_{i}(x)$ as the cyclic shift of $x$ by
$i$, that is, the sequence of length $k$ starting at 0
whose $j$-th increment is $\Delta
x(i+j)$ with $i+j$ interpreted  \text{mod} $k$. For any $x \in
\mathbb{B}^{k}$, let $\tau_{x}$ be the first time that $x$ hits its overall
minimum, i.e., $\min_{1 \leq i \leq k} x(i)$. 
The
(discrete form of) Vervaat's transformation
of $x$ is defined by $V(x) =
\omega_{\tau_{x}}(x)$; see \cite[Exercise 6.1.1]{Pitman2006} or
\cite{Takacs1962}. This transformation maps the set of bridges
$\mathbb{B}^{k}$ to the set of excursions of size $k$ finishing at $-1$:
\begin{eqnarray}
\mathbb{E}^{k} := \bigset{ (w(0), w(1), \dots, w(k)) \in \mathbb{Z}^{k+1}:
  \text{$w(0) = 0$,  $\Delta w(i) \geq -1$  for $1 \leq i \leq k$,
  and $w$ first hits $-1$ at time $k$}}. 
\end{eqnarray}
Let $\mathbf{n} = (n(i))_{i \geq 0}$ be a degree statistic with associated 
degree sequence $c(\mathbf{n}) = (c(i))_{i \geq 1}$, that is, the sequence
obtained by writing $n(0)$ zeros, $n(1)$ ones, and so on. Let $\sigma$ be a
uniformly random permutation of $\{1, \dots, |\mathbf{n}|\}$. Define the
bridge $\Wbn\in\bbB^{|\bn|}$ by letting 
\begin{eqnarray}
\Wbn(0) := 0 \hspace*{2mm} \text{and} \hspace*{2mm} \Wbn(j) := \sum_{i=1}^{j}(c(\sigma(i))-1), \hspace*{2mm} \text{for} \hspace*{2mm} 1 \leq j \leq |\mathbf{n}|. 
\end{eqnarray}
Note that $\Wbn(|\mathbf{n}|) = -1$.  $\Wbn$ is a discrete random process with exchangeable increments. The set of paths taken by $\Wbn$ is
\begin{eqnarray}
\mathbb{B}_{\mathbf{n}} := \bigset{ (x(0), x(1), \dots, x(|\mathbf{n}|)) \in \mathbb{B}^{|\mathbf{n}|}: |\{1 \leq j \leq |\mathbf{n}|: \Delta x(j) = i-1 \}| = n(i), \, \, \text{for every} \, \, i \geq 0 }. 
\end{eqnarray}
\noindent From the excursions in $\mathbb{E}^{|\mathbf{n}|}$, we consider those with fixed number of increments of given size, i.e., 
\begin{align}\label{tw5}
\mathbb{E}_{\mathbf{n}} 
:=\bbE^{|\bn|}\cap\bbB_\bn
= \bigset{ (w(0), w(1), \dots, w(|\mathbf{n}|)) \in \mathbb{E}^{|\mathbf{n}|}: |\{1 \leq j \leq |\mathbf{n}|: \Delta w(j) = i-1 \}| = n(i), \, \, \text{for every} \, \, i \geq 0 }. 
\end{align}
\noindent It is well known that there exists a bijection between
$\mathbb{E}_{\mathbf{n}}$ and $\mathbb{T}_{\mathbf{n}}$ 
(see \cite[Lemma 6.3]{Pitman2006}), and it is also well known that 
(see \cite[Exercise 6.2.1]{Pitman2006}) 
\begin{eqnarray}\label{sw4}
|\mathbb{T}_{\mathbf{n}}| = \frac{1}{|\mathbf{n}|} \binom{|\mathbf{n}|}{\mathbf{n}} = \frac{1}{|\mathbf{n}|} \frac{|\mathbf{n}|!}{\prod_{i \geq 0} n(i)!}.
\end{eqnarray}
\noindent It should be clear that bridges in $\mathbb{B}_{\mathbf{n}}$ are sent to excursions in $\mathbb{E}_{\mathbf{n}}$ by the Vervaat transformation. Moreover, for $w \in \mathbb{E}_{\mathbf{n}}$, the number of $x \in \mathbb{B}_{\mathbf{n}}$ such that $V(x) = w$ is exactly $|\mathbf{n}|$; see \cite[Corollary 15.4]{Svante2012} or \cite[Lemma 6.1]{Pitman2006}.

Let $u(1) \prec \cdots \prec u(|T|)$ be the sequence of vertices of $T \in
\mathbb{T}$ in depth-first order (also called pre-order); 
thus $u(1)$ is the root of $T$.  Set $d_{T}(i) 
= d_{T}(u(i))$, for $i = 1, \dots, |T|$, and call $(d_{T}(1), \dots,
d_{T}(|T|))$ the pre-order degree sequence of $T$. For $k \in \mathbb{N}$,
it is well-known that a sequence $(d(1), \dots, d(k)) \in
\mathbb{N}_{0}^{k}$ is the pre-order degree sequence of a tree $T \in
\mathbb{T}$ if and only if 
\begin{eqnarray} \label{eq1}
\sum_{i=1}^{j} d(i) \geq j, \hspace*{2mm} \text{for} \hspace*{2mm} 1 \leq j \leq k-1, \hspace*{2mm} \text{and} \hspace*{2mm} \sum_{i=1}^{k}d(i) = k-1;
\end{eqnarray}
see \cite[Lemma 15.2]{Svante2012}. Indeed, $T$ is uniquely determined by its
pre-order degree sequence. 
The depth first queue process (DFQP, or {\L}ukasiewicz path)
$W_{T} = (W_{T}(i), 0 \leq i \leq |T|)$ of a tree $T \in \mathbb{T}$,
associated to the depth-first ordering $u(1) \prec \dots \prec u(|T|)$ of
its vertices, is defined by $W_{T}(0) := 0$ and $W_{T}(i) := W_{T}(i-1)  +
d_{T}(i)-1$, for $1 \leq i \leq |T|$. Note that $\gD W_{T}(i)  =
d_{T}(i) -1 \geq -1$ for every $1 \leq i \leq |T|$, with equality if and
only if $u(i)$ is a leaf of $T$. In addition, $W_{T}(i) \geq 0$ for every $0
\leq i \leq |T|-1$, but $W_{T}(|T|) = -1$, 
i.e., $W_T\in\bbE^{|T|}$.

The next (well known)
proposition summaries some of the previous definitions and remarks.
\begin{proposition} \label{Pro1}
Let $\mathbf{n}$ be a degree statistic and let
$\mathcal{T}_{\mathbf{n}} \sim {\rm Unif}(\mathbb{T}_{\mathbf{n}})$. If\/
$W_{\mathcal{T}_{\mathbf{n}}}$ is the DFQP of $\mathcal{T}_{\mathbf{n}}$,
then $\mathbb{P}(W_{\mathcal{T}_{\mathbf{n}}} = w) =
\frac{1}{|\mathbb{T}_{\mathbf{n}}|}$, for $w \in \mathbb{E}_{\mathbf{n}}$. 
Moreover, if
$U$ is a uniform random variable on $\{1, \dots, |\mathbf{n}|\}$ independent
of $W_{\mathcal{T}_{\mathbf{n}}}$, then $(W_{\mathcal{T}_{\mathbf{n}}}, U)
\stackrel{\rm d}{=} (V(\Wbn), \tau_{\Wbn})$. 
\qed
\end{proposition}

Note, in particular, that $\tau_{\Wbn}$ is uniform on $\{1, \dots, |\mathbf{n}|\}$ and independent of $V(\Wbn)$.

\section{Moment computations} \label{Smoments}

In this section, we compute the moments of the number of fringe subtrees of
a uniformly random tree $\mathcal{T}_{\mathbf{n}}$ of
$\mathbb{T}_{\mathbf{n}}$, for a degree statistic $\mathbf{n}$. As a
warm-up, we use some of the ideas used
in \cite{Svante2016} to compute the first moment. 

Recall that $T \in \mathbb{T}$ is uniquely described by its pre-order degree sequence $(d_{T}(1), \dots, d_{T}(|T|))$. Then, for $i = 1, \dots, |T|$, the fringe subtree $T_{u(i)}$ has pre-order degree sequence $(d_{T}(i), \dots, d_{T}(i+k-1))$, where $1 \leq k \leq |T| -i+1$ is the unique index such that $(d_{T}(i), \dots, d_{T}(i+k-1))$ is a pre-order degree sequence of a tree, i.e., it satisfies (\ref{eq1}). Thus, for  $T, T^{\prime} \in \mathbb{T}$, we can write (\ref{eq2}) as
\begin{align}  \label{eq3}
N_{T^{\prime}}(T) =
              \sum_{i=1}^{|T| - |T^{\prime}|+1} \mathbf{1}_{\{ d_{T}(i) =
                              d_{T^{\prime}}(1), \dots,
                              d_{T}(i+|T^{\prime}|-1)) =
                              d_{T^{\prime}}(|T^{\prime}|)\}} 
,\end{align}
where the sum is interpreted as $0$ when $|T'|>|T|$.

\begin{lemma} \label{lemma1}
Let $\mathbf{n}$ be a degree statistic and let 
$\mathcal{T}_{\mathbf{n}} \sim {\rm Unif}(\mathbb{T}_{\mathbf{n}})$. 
For $T \in \mathbb{T}$ such that $|\mathbf{n}| \geq |T|$,
\begin{align} \label{eq5}
\E[N_{T}(\mathcal{T}_{\mathbf{n}})] =  \frac{|\mathbf{n}|}{(|\mathbf{n}|)_{|T|}} \prod_{i\geq 0} (n(i))_{n_{T}(i)}. 
\end{align} 
\end{lemma}

\begin{proof}
If  $ n_{T}(i) > n(i)$ for some $i \geq 0$, then 
both sides of \eqref{eq5} are $0$. 
Assume therefore that $n_{T}(i) \leq n(i)$ for all $i \geq 0$. 
Then, $|T| \leq |\mathcal{T}_{\mathbf{n}}| =|\mathbf{n}|$. 
Let $W_{\mathcal{T}_{\mathbf{n}}}$ be the DFQP of $\mathcal{T}_{\mathbf{n}}$ and $(d_{\mathcal{T}_{\mathbf{n}}}(1), \dots, d_{\mathcal{T}_{\mathbf{n}}}(|\mathbf{n}|))$ its pre-order degree sequence. Note that $d_{\mathcal{T}_{\mathbf{n}}}(i) = \Delta W_{\mathcal{T}_{\mathbf{n}}}(i)+1$, for $i = 1, \dots, |\mathbf{n}|$. Let $(d_{T}(1), \dots, d_{T}(|T|))$ be the pre-order degree sequence of $T$. Hence, by (\ref{eq3}),
\begin{eqnarray} \label{eq8}
N_{T}(\mathcal{T}_{\mathbf{n}}) = \sum_{i=1}^{|\mathbf{n}| - |T|+1} \mathbf{1}_{\{ \Delta W_{\mathcal{T}_{\mathbf{n}}}(i)= d_{T}(1)-1, \dots, \Delta W_{\mathcal{T}_{\mathbf{n}}}(i+|T|-1) = d_{T}(|T|)-1\}}.
\end{eqnarray}

For $1 \leq i \leq |\mathbf{n}| - |T|+1$ and $(y(1) \dots, y(i-1), y(i + |T|), \dots, y(|\mathbf{n}|)) \in \mathbb{Z}^{|\mathbf{n}| -|T|}$, define the walk 
\begin{eqnarray}
w^{y}_{i,T}(0) = 0 \hspace*{2mm} \text{and} \hspace*{2mm} w^{y}_{i,T}(j) = \sum_{r=1}^{j}(y(r)-1), \hspace*{2mm} \text{for} \hspace*{2mm} 1 \leq j \leq |\mathbf{n}|, 
\end{eqnarray}

\noindent where $y(i) = d_{T}(1), \dots, y(i+|T|-1) =d_{T}(|T|)$. 
In particular, $\Delta w^{y}_{i,T}(i+j-1) = d_{T}(j)-1$, for 
$j = 1, \dots, |T|$. We then consider the set of admissible excursions
obtained in this way:
\begin{eqnarray}\label{eq89}
\mathbb{A}_{i,T} = \{w^{y}_{i,T}: (y(1) \dots, y(i-1), y(i + |T|), \dots,
  y(|\mathbf{n}|)) \in \mathbb{Z}^{|\mathbf{n}| -|T|} 
\text{ and } w^{y}_{i,T} \in \mathbb{E}_{\mathbf{n}}\},
\end{eqnarray}
\noindent i.e., $\mathbb{A}_{i,T}$ is the set of excursions in
$\mathbb{E}_{\mathbf{n}}$ that code trees in $\mathbb{T}_{\mathbf{n}}$ with
a fringe subtree with pre-order degree sequence $(d_{T}(1), \dots,
d_{T}(|T|))$ that is rooted at its $i$-th vertex in depth-first order. 
Let $\tilde{\mathbf{n}} = (\tilde{n}(i))_{i \geq 0}$ be given by
$\tilde{n}(0) = n(0) - n_{T}(0)+1$ and $\tilde{n}(i) = n(i) - n_{T}(i)$ for
$i \geq 1$. 
If we instead of inserting the degree sequence of $T$ as above, insert only
$y(i)=0$ (corresponding to a leaf), and then relabel $y(j+|T|)$ as $y(j+1)$
for $i\le j\le |\bn|-|T|$, 
we obtain a bijection between $\mathbb{A}_{i,T}$ and the excursions in
$\bbE_{\tilde{\bn}}$ that correspond to a tree with a leaf as its $i$-th vertex.
Thus,
due to the bijection between $\mathbb{E}_{\tilde{\mathbf{n}}}$ and
$\mathbb{T}_{\tilde{\mathbf{n}}}$, 
we see that  
\begin{eqnarray} \label{eq9}
\sum_{i=1}^{|\mathbf{n}| - |T|+1} |\mathbb{A}_{i,T}| 
= |\mathbb{T}_{\tilde{\mathbf{n}}}|\cdot \tilde{n}(0).
\end{eqnarray} 
By Proposition \ref{Pro1} and (\ref{eq8})--\eqref{eq9}, 
this yields
\begin{eqnarray}
\E[N_{T}(\mathcal{T}_{\mathbf{n}})] 
= \sum_{i=1}^{|\mathbf{n}| - |T|+1} \frac{|\mathbb{A}_{i,T}|}{|\mathbb{T}_{\mathbf{n}}|} 
=\frac{|\mathbb{T}_{\tilde{\mathbf{n}}}|\cdot\tilde{n}(0)}{|\mathbb{T}_{\mathbf{n}}|}
, 
\end{eqnarray}
and the result \eqref{eq5} follows by \eqref{sw4}.
\end{proof}
 
Lemma \ref{lemma1} can be generalized to joint factorial moments of the
random variables $N_{T_{1}}(\mathcal{T}_{\mathbf{n}}), \dots,
N_{T_{m}}(\mathcal{T}_{\mathbf{n}})$, for $m \geq 1$ and a sequence of
distinct
rooted plane trees $T_{1}, \dots, T_{m} \in \mathbb{T}$. Before that, we need to introduce some notation. For $1 \leq i,j \leq m$, let
\begin{align}\label{ea3}
  \tau_{ij} \coloneqq N_{T_i}(T_j) \mathbf{1}_{\{ i \neq j \}}
\end{align}
\noindent be the number of proper fringe subtrees of $T_j$ that are equal to $T_i$. 
(Note that many of these terms are $0$. In particular, if we order $T_1, \dots, T_m$
according to their sizes, the matrix $(\tau_{ij})_{i,j=1}^{m} $ 
is strictly triangular.) 

For $q_{1}, \dots, q_{m} \in \mathbb{N}_{0}$, note that the product
  $(N_{T_{1}}(\mathcal{T}_{\mathbf{n}}))_{q_{1}} \cdots
  (N_{T_{m}}(\mathcal{T}_{\mathbf{n}}))_{q_{m}}$ is the number of sequences
  of $q \coloneqq q_1+\dots+q_m$ distinct fringe subtrees of
  $\mathcal{T}_{\mathbf{n}}$, where the first $q_1$ are copies of $T_1$, the
  next $q_2$ are copies of $T_2$, and so on.  
Given such a sequence of fringe subtrees, we say that these fringe subtrees are \emph{marked}. Furthermore, for each such sequence of marked fringe subtrees of $\mathcal{T}_{\mathbf{n}}$, say that a tree in the sequence is \emph{bound} if it is a fringe subtree of another tree in the sequence; otherwise it is \emph{free}. Note that the free trees are disjoint. Furthermore, each bound tree in the sequence is a fringe subtree of exactly one free tree. For a sequence $b=(b_{1}, \dots, b_{m}) \in \mathbb{N}_{0}^{m}$, let $S_{b}(\mathcal{T}_{\mathbf{n}})$ be the number of such sequences of $q$ fringe trees such that exactly $b_{i}$ of the fringe trees $T_i$ are bound, for $1 \leq i \leq m$. We thus have
\begin{align}\label{ea2}
\E\bigsqpar{(N_{T_{1}}(\mathcal{T}_{\mathbf{n}}))_{q_{1}} \dotsm
  (N_{T_{m}}(\mathcal{T}_{\mathbf{n}}))_{q_{m}}} 
= \sum_{b\in\bbNo^m}\E[S_{b}(\mathcal{T}_{\mathbf{n}})].
\end{align}
The sum is really only over $b=(b_{1}, \dots, b_{m}) \in \mathbb{N}_{0}^{m}$ such
  that $0 \leq b_{i} \leq q_{i}$ for $1 \leq i \leq m$, since otherwise $S_b(\cTn)=0$.
This sum can be computed by the following lemma.

\begin{lemma} \label{lemma4}
Let $\mathbf{n}$ be a degree statistic and let $\mathcal{T}_{\mathbf{n}}
\sim {\rm Unif}(\mathbb{T}_{\mathbf{n}})$. For $m \geq 1$ and $q_{1}, \dots,
q_{m} \in \mathbb{N}$, let $T_{1}, \dots, T_{m} \in \mathbb{T}$ be a
sequence of distinct
rooted plane trees such that $|\mathbf{n}| \geq \sum_{i=1}^{m}(q_{i} |T_{i}| -q_{i})+1$. Then
\begin{align}\label{ea6}
\E[S_{b}(\mathcal{T}_{\mathbf{n}})] = \frac{|\mathbf{n}|}{(|\mathbf{n}|)_{1+\sum_{j=1}^{m} (q_{j}-b_{j})(|T_{j}|-1)}} \prod_{i \geq 0}(n(i))_{\sum_{j=1}^{m} (q_{j}-b_{j})n_{T_{j}}(i)} \prod_{j=1}^{m} \frac{(q_j)_{b_j} \left( \sum_{k=1}^{m}(q_k-b_k)\tau_{jk} \right)_{b_j} }{b_{j}!}, 
\end{align}
for every $b=(b_{1}, \dots, b_{m}) \in \mathbb{N}_{0}^{m}$ such that $0 \leq
b_{i} \leq q_{i}$, for $1 \leq i \leq m$.
\end{lemma}
\begin{proof}
If  $\sum_{j=1}^{m}(q_{j}-b_{j})n_{T_{j}}(i)> n(i)$ for some $i \geq 0$, then 
both sides of \eqref{ea6} are $0$. We may thus assume that
$\sum_{j=1}^{m}(q_{j}-b_{j})n_{T_{j}}(i)\leq n(i)$ for all $i \geq 0$.

First, let us consider the case when all fringe trees are free, that is, 
the case $b =0= (0,\dots, 0) \in \mathbb{N}_{0}^{m}$. Replace each marked fringe subtree in $\mathcal{T}_{\mathbf{n}}$ by a single leaf; moreover, mark this leaf and order all marked leaves into a sequence, corresponding to the order of the fringe subtrees. This yields another tree $\Ttn$, which we call a \emph{reduced tree}, with a sequence of $q$ marked leaves. Since $\mathcal{T}_{\mathbf{n}}$ has $n(i)$ vertices of degree $i$, for $i \geq 0$, and we have replaced $q_j$ copies of $T_j$ by leaves, the degree statistic $\tilde{\mathbf{n}} = (\tilde{n}(i))_{ i \geq 0}$ of $\Ttn$ is given by
\begin{align}  \label{ea1}
\tilde{n}(i) \coloneqq 
  \begin{cases}
             n(i) - \sum_{j=1}^{m} q_{j}n_{T_{j}}(i),  & i \geq 1,\\
             n(0) - \sum_{j=1}^{m} q_{j}n_{T_{j}}(0) + \sum_{j=1}^{m}q_{j},  & i=0,
  \end{cases}
\end{align}
\noindent and has size
\begin{align}\label{ea8}
  |\tilde{\mathbf{n}}| \coloneqq \sum_{i \geq 0} \tilde{n}(i)  = |\mathbf{n}| - \sum_{j=1}^{m}q_{j}(|T_{j}|-1).
\end{align}
\noindent There is a one-to-one correspondence between trees in $\mathbb{T}_{\mathbf{n}}$ with
a sequence of marked fringe subtrees as above, and reduced trees with the degree statistic \eqref{ea1} and a sequence of $q$ marked leaves. If we ignore the marks, the number of possible reduced trees is given by \eqref{sw4} with the degree statistic $\tilde{\mathbf{n}}$ in \eqref{ea1}. In each unmarked reduced tree, the number of ways to choose sequences of marked leaves is $(\tilde{n}(0))_{q_1+\dots+q_m}$. Thus, the number of trees in $\mathbb{T}_{\mathbf{n}}$ with marked sequences of free fringe subtrees is the product of these numbers, i.e.,
\begin{align} \label{ea7}
\frac{(|\tilde{\mathbf{n}}|-1)!}{\prod_{i \geq 0}\tilde{n}(i)!} (\tilde{n}(0))_{\sum_{j=1}^{m}q_{j}} = \frac{(|\tilde{\mathbf{n}}|-1)!}{\prod_{i \geq 0}(n(i) - \sum_{j=1}^{m} q_{j}n_{T_{j}}(i))!}.
\end{align}
\noindent By dividing with $|\mathbb{T}_{\mathbf{n}}|$, which is given by \eqref{sw4}, and using \eqref{ea8}, we find
\begin{align} \label{ea4}
\E[S_{0}(\mathcal{T}_{\mathbf{n}})] = \frac{1}{(|\mathbf{n}|-1)_{\sum_{j=1}^{m} q_{j}(|T_{j}|-1)}} \prod_{i \geq 0}(n(i))_{\sum_{j=1}^{m} q_{j}n_{T_{j}}(i)}.
\end{align}

Now consider the general case with a sequence $b=(b_{1}, \dots, b_{m})$ telling the number of bound fringe subtrees. There are thus $q_j-b_j$ free trees of type $T_j$. The number of ways to choose the positions of the bound trees in the sequences of fringe trees is $\prod_{j=1}^{m}\binom{q_j}{b_j}$, and for each choice of free trees, there are $\sum_{k=1}^{m} (q_k-b_k) \tau_{jk}$ possible bound trees of type $T_j$; thus the number of choices of the bound trees is
\begin{align}\label{ea5}
\prod_{j=1}^{m} \frac{(q_j)_{b_j} \left( \sum_{k=1}^{m}(q_k-b_k)\tau_{jk} \right)_{b_j} }{b_{j}!}.
\end{align}
\noindent The number of trees in $\mathbb{T}_{\mathbf{n}}$ with sequences of $q_j-b_j$ free trees $T_j$, for $1\le j\le m$, is given by replacing $q_j$ by $q_j-b_j$ in \eqref{ea1}--\eqref{ea7}. Hence, we obtain \eqref{ea6}, extending \eqref{ea4}.
\end{proof}

We record two important special cases of the computation above.

\begin{lemma}  \label{LemC}
Let $\mathbf{n}$ be a degree statistic and let 
$\mathcal{T}_{\mathbf{n}} \sim {\rm Unif}(\mathbb{T}_{\mathbf{n}})$. 
\begin{enumerate}[label=\upshape(\roman*)]
\item \label{LemC1}
For $q \in \mathbb{N}$ and $T \in \mathbb{T}$ such that $|\mathbf{n}|\geq q|T|-q+1$,
\begin{align}  \label{eq3''}
\E[(N_{T}(\mathcal{T}_{\mathbf{n}}))_{q}]  =  \frac{|\mathbf{n}|}{(|\mathbf{n}|)_{q|T|-q+1}} \prod_{i\geq 0} (n(i))_{qn_{T}(i)}. 
\end{align} 
\item \label{LemC2}
For distinct
$T, T^{\prime} \in \mathbb{T}$ such that $|\mathbf{n}| \geq |T|+|T^{\prime} |-1$, 
\begin{align} \label{eq4}
\E[N_{T}(\mathcal{T}_{\mathbf{n}})N_{T^{\prime}}(\mathcal{T}_{\mathbf{n}})]  
=   N_{T}(T^{\prime} ) \E[N_{T^{\prime} }(\cTn)]
+ N_{T^{\prime} }(T)\E[N_{T}(\cTn)] +  \frac{|\mathbf{n}|}{(|\mathbf{n}|)_{|T| +|T^{\prime}|-1}} \prod_{i\geq 0} (n(i))_{n_{T}(i) + n_{T^{\prime}}(i)}. 
\end{align}  
\end{enumerate}
\end{lemma}

\begin{proof}
\pfitemref{LemC1}
This is the case $m=1$ of \eqref{ea2} and \eqref{ea6}, when we consider only one
tree $T_1$. In this case, there are no bound fringe trees, and thus we only
have to consider $b=0$ in \eqref{ea2}. Taking $b_1=0$ (and $q_1=q$) in
\eqref{ea6} yields \eqref{eq3''}.

\pfitemref{LemC2} 
This is the case $m=2$ and $q_1=q_2=1$ of \eqref{ea2}.
The possible vectors $b=(b_1,b_2)$ are $(1,0)$, $(0,1)$, and $(0,0)$,
and it is easily verified that taking these three vectors
in \eqref{ea6}, and using \eqref{eq3''} with $q=1$ in two cases, 
yields the three terms
on the right-hand side of \eqref{eq4}.
\end{proof}

\section{Proof of Theorems \ref{TheoLLN} and \ref{TA}}
\label{SpfLLN+A}

In this section we prove Theorems  \ref{TheoLLN} and \ref{TA} (in
opposite order).
In  what follows we will frequently use the following well-known estimate.   

\begin{lemma}\label{L0}
If\/ $x \geq 1$ is a real number and $0 \leq k \leq x/2$ is an integer, then
\begin{align}
(x)_{k} = x^{k} \exp \left( -\frac{k(k-1)}{2x} +O\left(\frac{k^3}{x^2} \right) \right).
\end{align}
\end{lemma}
\begin{proof}
Since $\ln(1-y)=-y+O(y^2)$ for $0\leq y\leq 1/2$, the result follows from the identity
\begin{align}
\frac{(x)_k}{x^k} =\prod_{i=0}^{k-1} \frac{x-i}{x} =\exp\left(\sum_{i=0}^{k-1} \ln\left(1-\frac{i}{x} \right)\right).
\end{align}
\end{proof}

\begin{proof}[Proof of \refT{TA}]
  Note first the trivial bound 
  \begin{align}\label{ta2}
    N_T(\cTn) \le \frac{n(i)}{n_T(i)} \le n(i),
\qquad i\in\cD(T),
  \end{align}
since the copies of $T$ in $\cTn$ are distinct.
Furthermore, by \eqref{pip} and \eqref{pin}, 
\begin{align}\label{ta3}
 |\bn|\pi_\bpn(T) \le |\bn| p_i(\bn) = n(i),
\qquad 
i\in\cD(T).
\end{align}
Hence, \eqref{ta1} is trivial if $n(i)=O(1)$ for some $i\in\cD(T)$.
In particular, we may in the sequel assume $n(i)\ge 2 n_T(i)$ for every $i \geq 0$,
and thus $|\bn|\ge 2|T|$.
Then, by \eqref{eq5} and \refL{L0},
\begin{align}\label{ta4}
\E N_T(\cTn)& 
=|\bn|^{1-|T|}\prod_{i\in\cD(T)} n(i)^{n_T(i)}
\cdot\exp\lrpar{\frac{|T|(|T|-1)}{2|\bn|}-\sum_{i\in\cD(T)}
\frac{n_T(i)(n_T(i)-1)}{2n(i)}+O\Bigpar{\sum_{i\in\cD(T)}\frac{1}{n(i)^2}}}
\notag\\&
=\ENTX 
\cdot\exp\lrpar{\frac{|T|(|T|-1)}{2|\bn|}-\sum_{i\in\cD(T)}
\frac{n_T(i)(n_T(i)-1)}{2n(i)}+O\Bigpar{\sum_{i\in\cD(T)}\frac{1}{n(i)^2}}}
,\end{align}
which implies \eqref{ta1} by \eqref{ta3}.

Similarly, taking $q=2$ in \eqref{eq3''}, and now assuming as we may $n(i)\ge
4n_T(i)$ for every $i \geq 0$,
\begin{align}\label{ta11}
\E \xpar{N_T(\cTn)}_2& 
=  \frac{|\bn|}{(|\bn|)_{2|T|-1}} \prod_{i\in\cD(T)} (n(i))_{2n_{T}(i)}
\notag\\&
=|\bn|^{2-2|T|}\prod_{i\in\cD(T)} n(i)^{2n_T(i)}
\cdot\exp\lrpar{\frac{(2|T|-1)(2|T|-2)}{2|\bn|}-\sum_{i\in\cD(T)}
\frac{2n_T(i)(2n_T(i)-1)}{2n(i)}+O\Bigpar{\sum_{i\in\cD(T)}\frac{1}{n(i)^2}}}
\notag\\&
=
\bigpar{|\bn|\pi_\bpn(T)}^2
\cdot\exp\lrpar{\frac{(2|T|-1)(|T|-1)}{|\bn|}-\sum_{i\in\cD(T)}
\frac{n_T(i)(2n_T(i)-1)}{n(i)}+O\Bigpar{\sum_{i\in\cD(T)}\frac{1}{n(i)^2}}}
,\end{align}
Hence, using also \eqref{ta4},
\begin{align}  \label{ta12}
\E \xpar{N_T(\cTn)}_2& 
=\bigpar{\E N_T(\cTn)}^2
\cdot\exp\lrpar{\frac{(|T|-1)^2}{|\bn|}-\sum_{i\in\cD(T)}
\frac{n_T(i)^2}{n(i)}+O\Bigpar{\sum_{i\in\cD(T)}\frac{1}{n(i)^2}}}
.
\end{align}
Consequently,
using \eqref{ta1} and noting that $\E N_T(\cTn) = O(n(i))$ for $i\in\cD(T)$
by \eqref{ta1} and \eqref{ta3}, 
\begin{align}\label{ta13}
  \Var \sqpar{N_T(\cTn)}&
= \E \xpar{N_T(\cTn)}_2 + \E N_T(\cTn)
- \bigpar{\E N_T(\cTn)}^2
\notag\\&
=\bigpar{\E N_T(\cTn)}^2
\cdot\lrpar{\frac{(|T|-1)^2}{|\bn|}-\sum_{i\in\cD(T)}
\frac{n_T(i)^2}{n(i)}}
+ \E N_T(\cTn) + O(1)
\notag\\&
=\bigpar{\ENTX}^2
\cdot\lrpar{\frac{(|T|-1)^2}{|\bn|}-\sum_{i\in\cD(T)}
\frac{n_T(i)^2}{n(i)}}
+ \ENTX + O(1)
,\end{align}
which yields \eqref{ta10} by the definitions \eqref{eq20}, \eqref{eq12} and
\eqref{pin}.

For the proof of \eqref{ta20} we use \eqref{eq4}. 
The first two terms are handled by \eqref{ta1}, and the final term is
treated as in \eqref{ta11}--\eqref{ta13} with mainly notational differences;
we omit the details.
\end{proof}

\begin{proof}[Proof of Theorem \ref{TheoLLN}]
By Condition \ref{Condition1},
we have $p_i(\bnk)\to p_i$ for every $i\ge0$, and thus
$\pi_\bpnk(T)\to\pi_\bp(T)$.
Hence, \ref{Annealed} follows from \eqref{ta1}.

Moreover, it follows from \eqref{eq12}--\eqref{eq20} that 
$\gamma_\bpnk(T,T)=O(1)$ (for a fixed $T$),
and thus  \eqref{ta10} yields $\Var N_T(\Tnk) = O(|\bnk|)$.
Therefore, \ref{Quenched} follows from \ref{Annealed} and Chebyshev's
inequality.
\end{proof}

\section{Proof of  Theorems \ref{TheoCLT} and \ref{TheoCLT2}}
\label{SpfCLT}

We have now all the ingredients to prove
Theorems \ref{TheoCLT} and Theorem \ref{TheoCLT2}. 
Theorem \ref{TheoCLT} is essentially a  special case
\refT{TheoCLT2}, combined with the already proved \refT{TA}.
Nevertheless, in order to focus on the main ideas,
we give first a separate proof of \refT{TheoCLT},
and then  the rather small modifications required for the 
more technical general version in Theorem \ref{TheoCLT2}.

\begin{proof}[Proof of Theorem \ref{TheoCLT}]
First note that 
\refCond{Condition1} implies 
\begin{align}\label{ju1}
\pi_{\bpnk}(T_i)\to\pip(T_i)
\quad\text{and}\quad
\gamma_\bpnk(T_i,T_j)\to\gamma_\bp(T_i,T_j), \qquad \text{for }1\le i,j\le m.
\end{align}
Hence,
\eqref{eq6}--\eqref{eq13} follow from
\eqref{ta1}--\eqref{ta20} in \refT{TA}.

We next prove the asymptotic normality result in (\ref{eq7}). 
Note first that \eqref{ta1} implies that it does not matter whether we use 
$\E[N_{T_{i}}(\mathcal{T}_{\mathbf{n}_{\kappa}})]$ 
or $\mu_\bnk(T_i)=|\mathbf{n}_{\kappa}| \pi_{\mathbf{p}(\mathbf{n}_{\kappa})}(T_{i})$
in \eqref{eq7}.

If $\pi_{\mathbf{p}}(T_{i}) = 0$, for some $1 \leq i \leq m$, 
then it follows from 
\eqref{eq20} that $\gammap(T_i,T_i)=0$, and thus \eqref{eq15} yields
$\Var[N_{T_i}(\Tnk)]=o(|\bnk|)$; consequently,
\eqref{ta1} and Chebyshev's inequality yield, as $\kappa \rightarrow \infty$,
\begin{eqnarray}\label{ju9}
\frac{N_{T_{i}}(\mathcal{T}_{\mathbf{n}_{\kappa}}) - \E[N_{T_{i}}(\mathcal{T}_{\mathbf{n}_{\kappa}})]}{\sqrt{|\mathbf{n}_{\kappa}|}} \pto 0.
\end{eqnarray}
Hence, convergence of the $i$-th component in \eqref{eq7} is trivial in this case.
Furthermore, $\pip(T_i)=0$ also implies $\gammap(T_i,T_j)=0$ for every $1 \le j\le m$
by \eqref{eq21}, noting that if $N_{T_i}(T_j)>0$ then also $\pip(T_j)=0$.
Thus, we may ignore all $i$ in \eqref{eq7} with $\pip(T_i)=0$
and show (joint) convergence for the remaining ones,
because then (\ref{eq7}) in general will
follow from \cite[Theorem 3.9 in Chapter 1]{Billingsley1999}. 
Consequently, 
we henceforth
assume 
that $\pi_{\mathbf{p}}(T_{i}) > 0$ for
all $1 \leq i \leq m$. 
Equivalently, $p_{k} >0$ for every $k\in\bigcup_{i=1}^m \cD(T_i)$.
We may also assume that $T_1,\dots,T_m$ are distinct.

To see the main idea of the proof, consider first the univariate case $m=1$.
(The impatient reader may skip this and proceed directly to the general case.)
We omit the index 1 and write $T$ instead of $T_1$.
In this case, we can use the
Gao--Wormald theorem \cite[Theorem 1]{Gao2004}
and the following estimate.
For  any $q_\kk=O(|\bnk|\qq)$,
\eqref{eq3''} and \refL{L0} yield,
recalling the definitions \eqref{pin}, \eqref{pip}, 
\eqref{eq12}, \eqref{eq20}, and \eqref{mu}
 of $p_i(\bn)$, $\pip(T)$, $\eta_{\mathbf{p}}(T,T)$,  $\gammap(T,T)$, 
and $\mu_\bnk(T)$,
\begin{align}\label{jum1}
\E[(N_{T}(\mathcal{T}_{\mathbf{n}_{\kappa}}))_{q_{\kappa}}] 
&=
\frac{\prod_{i \geq 0}{n_\kk(i)}^{ q_{\kk} n_{T}(i)}}
{|\bnk|^ {q_{\kk}(|T|-1)}}
 \exp\lrpar{\frac{\bigpar{ q_{\kk}(|T|-1)}^2}{2|\bnk|}
-\sum_{i \geq 0}\frac{\bigpar{ q_{\kk}n_{T}(i)}^2}{2n_\kk(i)}
+o(1)}
\notag\\&
=|\bnk|^{q_{\kk}}\prod_{i \geq 0}
 p_i(\bnk)^{q_{\kk} n_{T}(i)}
\exp\lrpar{\frac{\bigpar{ q_{\kk}(|T|-1)}^2}{2|\bnk|}
-\sum_{i \geq 0}\frac{\bigpar{ q_{\kk} n_{T}(i)}^2}{2n_\kk(i)}
+o(1)}
\notag\\& 
= \bigpar{|\bnk|\pi_{\mathbf{p}(\mathbf{n}_{\kappa})}(T)}^{q_{\kappa}} 
\exp \left( \frac{q_{\kappa}^{2}}{2 |\mathbf{n}_{\kappa}|} 
\eta_{\mathbf{p}(\mathbf{n}_{\kappa})}(T,T) +  o(1) \right) 
\notag\\& 
= \mu_{\bnk}(T)^{q_{\kappa}} 
\exp \left( \frac{(\gamma_{\bpnk}(T,T)-\pi_{\bpnk}(T))|\bnk|}{2 \mu_{\bnk}(T)^2} q_{\kappa}^{2}
 +  o(1) \right) 
\notag\\& 
= \mu_{\bnk}(T)^{q_{\kappa}} 
\exp \left( \frac{\gammap(T,T)|\bnk|-\mu_{\bnk}(T)}{2 \mu_{\bnk}(T)^2} q_{\kappa}^{2}
 +  o(1) \right) 
.\end{align}
If $\gammap(T,T)>0$, we may now apply 
the Gao--Wormald theorem \cite[Theorem 1]{Gao2004} with
$\mu_\kk:=\mu_{\bnk}(T)$ and 
$\gs_\kk^2:=\gammap(T,T)|\bnk|$ and conclude \eqref{eq70},
which by \eqref{eq15} is equivalent to \eqref{eq7} (with $m=1$). 
The case $\gammap(T,T)=0$ is trivial, since then \eqref{eq15} implies
\eqref{ju9}. 
Alternatively, for any $\gammap(T,T)$,
we may take the same $\mu_\kk$ but $\gs_\kk^2:=|\bnk|$ in the
case $m=1$ of our version \refT{TGW} of the Gao--Wormald theorem.

The general case follows similarly by 
a multidimensional version of the Gao--Wormald theorem, 
which we state and prove as
Theorem \ref{TGW} in Appendix \ref{AppendixA}, 
together with the following  estimates of (joint)
factorial moments. The main complication in the multivariate case is the
possibility that fringe trees of type $T_j$ may contain fringe trees of type
$T_k$ for some $1\le j,k\le m$; we thus use the decomposition in \eqref{ea2}
and estimate the terms separately.

Write for convenience 
$\mu_{i \kappa} \coloneqq \mu_\bnk(T_i)=|\mathbf{n}_{\kappa}|
\pi_{\mathbf{p}(\mathbf{n}_{\kappa})} (T_{i})$,
for $1 \leq i \leq m$. Note that our assumption $\pip(T_i)>0$ and \eqref{ju1}
imply that (possibly ignoring some small $\kk$)
\begin{align}
  \label{ju2}
\mu_{i\kk}=\Theta(|\bnk|),
\qquad 1\le i\le m.
\end{align}
Furthermore, 
for every $i\in\bigcup_{j=1}^{m}\cD(T_j)$, we have  by assumption $p_i>0$,
and thus \refCond{Condition1} implies 
\begin{align}
  \label{ju3}
n_\kk(i)=\Theta(|\bnk|).
\end{align}

Let $q_{1\kappa}, \dots, q_{m \kappa} \in \mathbb{N}_{0}$ 
be such that 
\begin{align}\label{ju4}
q_{i \kappa} =O(|\bnk|\qq), \qquad  1 \leq i \leq m
.\end{align}

Let $b=(b_{1}, \dots, b_{m}) \in \mathbb{N}_{0}^{m}$ be a fixed
sequence.
Assume first that $q_{j \kappa}\ge b_j$ for $1 \leq j \leq m$. 
Then
we deduce from (\ref{ea6}) in Lemma \ref{lemma4} and \refL{L0},
using \eqref{ju2}--\eqref{ju4} and 
recalling
\eqref{pin},  
\eqref{pip} and \eqref{mu},
\begin{align}\label{er1}
\E S_b(\Tnk)&
=
\frac{\prod_{i \geq 0}{n_\kk(i)}^{\sum_{j=1}^{m} (q_{j\kk}-b_j) n_{T_j}(i)}}
{|\bnk|^{\sum_{j=1}^{m} (q_{j\kk}-b_j)(|T_j|-1)}}
  \prod_{j=1}^{m}\frac{(q_{j\kk}+O(1))^{b_j} 
  \bigpar{\sum_{k=1}^{m}(q_{k\kk}-b_k)\tau_{jk}+O(1)}^{b_j}}{b_j!}
\notag\\&\qquad\times
\exp\lrpar{\frac{\bigpar{\sum_{j=1}^{m} (q_{j\kk}-b_j)(|T_j|-1)}^2}{2|\bnk|}
-\sum_{i \geq 0}\frac{\bigpar{\sum_{j=1}^{m} (q_{j\kk}-b_j)n_{T_j}(i)}^2}{2n_\kk(i)}
+o(1)}
\notag\\&
=|\bnk|^{\sum_{j=1}^{m} (q_{j\kk}-b_j)}\prod_{i \geq 0}
 \prod_{j=1}^{m}p_i(\bnk)^{(q_{j\kk}-b_j) n_{T_j}(i)}
\prod_{j=1}^{m}\frac{ \bigpar{q_{j\kk}\sum_{k=1}^{m} q_{k\kk}\tau_{jk}+O(q_{j\kk}+\sum_{k=1}^{m}\tau_{jk}q_{k\kk})}^{b_j}}{b_j!}
\notag\\& 
\qquad\times
\exp\lrpar{\frac{\bigpar{\sum_{j=1}^{m} q_{j\kk}(|T_j|-1)}^2}{2|\bnk|}
-\sum_{i \geq 0}\frac{\bigpar{\sum_{j=1}^{m} q_{j\kk} n_{T_j}(i)}^2}{2n_\kk(i)}
+o(1)}
\notag\\& 
=|\bnk|^{\sum_{j=1}^{m} (q_{j\kk}-b_j)}\prod_{j=1}^{m} \pipnk(T_j)^{q_{j\kk}-b_j}
\prod_{j=1}^{m}\frac{ \bigpar{q_{j\kk}\sum_{k=1}^{m} q_{k\kk}\tau_{jk}+O(\mu_{j\kk}\qq)}^{b_j}}{b_j!}
\notag\\& 
\qquad\times
\exp\lrpar{\frac{\bigpar{\sum_{j=1}^{m} q_{j\kk}(|T_j|-1)}^2}{2|\bnk|}
-\sum_{i \geq 0}\frac{\bigpar{\sum_{j=1}^{m} q_{j\kk} n_{T_j}(i)}^2}{2n_\kk(i)}
+o(1)}
\notag\\&
=\prod_{j=1}^{m}\mujk^{q_{j\kk}}
\prod_{j=1}^{m}\frac{ \bigpar{q_{j\kk}\sum_{k=1}^{m} q_{k\kk}\tau_{jk}/\mujk+o(1)}^{b_j}}
{b_j!}
\notag\\& \qquad\times
\exp\lrpar{\frac{\bigpar{\sum_{j=1}^{m} q_{j\kk}(|T_j|-1)}^2}{2|\bnk|}
-\sum_{i \geq 0}\frac{\bigpar{\sum_{j=1}^{m} q_{j\kk} n_{T_j}(i)}^2}{2n_\kk(i)}
+o(1)}
.\end{align}
On the right-hand side,
by \eqref{ju2}--\eqref{ju4},
each factor in the second product  is $O(1)$, and so
is the exponential factor.
Consequently, \eqref{er1} yields
\begin{align}\label{er2}
\frac{\E S_b(\Tnk)}{\prod_{j=1}^{m}\mujk^{q_{j\kk}}}
=
\prod_{j=1}^{m}\frac{ \bigpar{q_{j\kk}\sum_{k=1}^{m} q_{k\kk}\tau_{jk}/\mujk}^{b_j}}
{b_j!}
\cdot
\exp\lrpar{\frac{\bigpar{\sum_{j=1}^{m} q_{j\kk}(|T_j|-1)}^2}{2|\bnk|}
-\sum_{i\geq 0}\frac{\bigpar{\sum_{j=1}^{m} q_{j\kk} n_{T_j}(i)}^2}{2n_\kk(i)}}
+o(1)
.\end{align}
This is trivially true also if $q_{j\kk}<b_j$ for some $1 \leq j \leq m$, since then
$S_b(\Tnk)=0$ and the first term on the right-hand side of \eqref{er2} then is
easily seen to be $o(1)$.
Hence, \eqref{er2} holds for every fixed $b\in\bbNo^m$, uniformly for all
$(q_{j\kk})_{j=1}^m$ that satisfy \eqref{ju4}.
 
Furthermore, a simple variant of this calculation shows that,  
for each constant $C >0$,
uniformly for all $\kk \geq 1$, $b=(b_{1}, \dots, b_{m})\in \mathbb{N}_{0}^{m}$ 
and  $q_{1\kk},\dots,q_{m\kk} \in \mathbb{N}_{0}$ 
such that 
\eqref{ju4} holds,
\begin{align} \label{ec2a}
\E S_{b}(\mathcal{T}_{\mathbf{n}_{\kappa}}) & 
\leq \frac{\prod_{i \geq 0}  n_{\kappa}(i)^{\sum_{j=1}^{m}(q_{j \kappa} - b_{j})n_{T_{j}}(i)}  }{|\mathbf{n}_{\kappa}|^{\sum_{j=1}^{m}(q_{j \kappa}-b_{j})(|T_{j}|-1)}} \prod_{j=1}^{m} \frac{q_{j \kappa}^{b_j} \left( \sum_{k=1}^{m} q_{k \kappa} \tau_{jk} \right)^{b_j} }{b_{j}!} \exp(O(1)) \nonumber \\
& \leq C_1 \prod_{j=1}^{m} \mu_{j\kappa}^{q_{j \kappa}} \prod_{j=1}^{m} \frac{C_2^{b_{j}}}{b_{j}!},
\end{align}
for some constants $C_1 >0$ and $C_2>0$ (depending on $C$).
Equivalently,
\begin{align} \label{ec2b}
\frac{\E[S_{b}(\mathcal{T}_{\mathbf{n}_{\kappa}})]}
{\prod_{j=1}^{m} \mu_{j\kappa}^{q_{j \kappa}}} 
\le C_1\prod_{j=1}^{m} \frac{C_2^{b_{j}}}{b_{j}!}.
\end{align}
By \eqref{ju2}--\eqref{ju4},
the same estimate  holds for
the first term on the right-hand side of \eqref{er2}, and thus it holds also
for the $o(1)$ term there.

Now sum \eqref{er2} over all $b=(b_{1}, \dots, b_{m}) \in \mathbb{N}_{0}^{m}$.
The sum of the error terms $o(1)$ is 
$o(1)$ by dominated convergence justified by \eqref{ec2b} and the comments
after it. 
Hence,
\eqref{ea2}
and \eqref{er2} yield, uniformly for all $q_{1\kk},\dots,q_{m\kk}$
satisfying \eqref{ju4}, 
recalling
$n_{\kappa}(i) = |\mathbf{n}_{\kappa}| p_{i}(\mathbf{n}_{\kappa})$ for $i \geq 0$, 
\begin{align} \label{eb2}
\frac{\E[(N_{T_{1}}(\mathcal{T}_{\mathbf{n}}))_{q_{1 \kappa}}
\dotsm (N_{T_{m}}(\mathcal{T}_{\mathbf{n}}))_{q_{m  \kappa}}]}
{\prod_{j=1}^{m} \mu_{j\kappa}^{q_{j \kappa}}}
&= 
\prod_{j=1}^{m} \exp 
\left( \frac{ q_{j \kappa} \sum_{k=1}^{m}q_{k \kappa} \tau_{jk}}{\mu_{j \kappa}} \right)
 \nonumber  \\& 
\qquad \times \exp \left(  \frac{\left(\sum_{j=1}^{m}q_{j \kappa}(|T_{j}|-1) \right)^{2}}{2 |\mathbf{n}_{\kappa}|} - \sum_{i\ge0} \frac{\left(\sum_{j=1}^{m} q_{j \kappa} n_{T_{j}}(i) \right)^{2}}{2 |\mathbf{n}_{\kappa}| p_{i}(\mathbf{n}_{\kappa}) } \right) + o(1) \nonumber  \\
& =  \exp \left( \frac{1}{2} \sum_{j,k=1}^{m} 
\frac{r_{\kappa}(j,k) |\bnk| }{\mu_{j\kappa} \mu_{k\kappa} }q_{j \kappa}q_{k \kappa} 
\right) + o(1),
\end{align}
where, by a simple calculation recalling \eqref{mu} and \eqref{eq12},
\begin{align} \label{eb3}
r_{\kappa}(j,k)  \coloneqq 2 \tau_{jk}
  \pi_{\mathbf{p}(\mathbf{n}_{\kappa})}(T_{k}) +
  \eta_{\mathbf{p}(\mathbf{n}_{\kappa})}(T_{j},T_{k})
  \pi_{\mathbf{p}(\mathbf{n}_{\kappa})}(T_{j})
  \pi_{\mathbf{p}(\mathbf{n}_{\kappa})}(T_{k}),
\qquad 1 \leq j,k \leq m. 
\end{align}
In \eqref{eb2}, we may replace $r_{\kappa}(j,k)$ by the symmetrization
$\tr_{\kappa}(j,k)\coloneqq \frac{1}{2} (r_{\kappa}(j,k) +
r_{\kappa}(k,j))$. 
Comparing \eqref{eb3} and \eqref{eq20}--\eqref{eq21}, using
\eqref{ea3} and treating the cases $j=k$ and $j\neq k$ separately,
we obtain
\begin{align}\label{gw8}
  \tr_\kk(j,k)=\gamma_\bpnk(T_j,T_k)-\gd_{jk}\pipnk(T_j),
\qquad 1\le j,k\le m
.\end{align}
Define $\gs_{j\kk}:=|\bnk|\qq$ for every $1 \leq j \leq m$;
then \eqref{eb2} yields,
using \eqref{gw8} and \eqref{ju1},
\begin{align}\label{ju6}
  \E\bigsqpar{(N_{T_{1}}(\mathcal{T}_{\mathbf{n}}))_{q_{1 \kappa}}
\dotsm (N_{T_{m}}(\mathcal{T}_{\mathbf{n}}))_{q_{m  \kappa}}}
=\prod_{j=1}^{m} \mu_{j\kappa}^{q_{j \kappa}}
 \exp \left( \frac{1}{2} \sum_{j,k=1}^{m} 
\frac{{\gamma_\bp(T_j,T_k)\gs_{j\kk}\gs_{k\kk}-\gd_{jk}\mu_{j\kk}} }{\mu_{j\kappa} \mu_{k\kappa} }q_{j \kappa}q_{k \kappa}+o(1) 
\right),
\end{align}
uniformly in all $q_{1\kk},\dots,q_{m\kk}$ that satisfy \eqref{ju4}.
We apply Theorem
\ref{TGW}, 
and note that \eqref{ju6} is the condition \eqref{tgw2} 
(with obvious changes of notation);
furthermore, by \eqref{ju1}, our choices
$\mu_{i \kappa}:=|\mathbf{n}_{\kappa}|\pi_{\mathbf{p}(\mathbf{n}_{\kappa})} (T_{i})$ 
and
$\sigma_{i \kappa} \coloneqq |\mathbf{n}_{\kappa}|^{1/2}$
satisfy \eqref{tgw1}.
Hence, Theorem \ref{TGW} yields \eqref{eq7}. 

Finally, the last assertion of Theorem \ref{TheoCLT} is proved in Lemma
\ref{L+1} below. 
\end{proof}

\begin{lemma}\label{L+1}
Suppose that $\bp=(p_i)_{i \geq 0}\in\ppi$, and
let $T\in\bbT$ with $|T|>1$. If $\pi_\bp(T)>0$, then $\gamma_\bp(T,T)>0$.
\end{lemma}
Recall that the assumption $\pi_\bp(T)>0$ is equivalent to
$p_i>0$ for all $i\in\cD(T)$.

\begin{proof}
  Suppose that $\pi_\bp(T)>0$ but $\gamma_\bp(T,T)=0$.
By \eqref{eq20} and \eqref{eq12}, this means
\begin{align}\label{tx1}
    0 
=\frac{\gammap(T,T)}{\pip(T)^2}
=\frac{1}{\pip(T)}+\etap(T,T)
=\frac{1}{\pip(T)}+{(|T|-1)^2-\sum_{i\in\cD(T)}\frac{n_T(i)^2}{p_i}}.
\end{align}
Furthermore, since \eqref{eq7} in \refT{TheoCLT} applies, 
for any $m\geq1$ and trees $T_1,\dots,T_m\in\bbT$,
the matrix
$(\gammap(T_i,T_j))_{i,j=1}^{m}$ is a covariance matrix, and thus positive semidefinite.
Hence, the Cauchy--Schwarz inequality holds for $\gammap(T_i,T_j)$.
In particular, for any tree $T'\in\bbT$,
\begin{align}\label{tx2}
|\gammap(T,T')|\le\gammap(T,T)\qq\gammap(T',T')\qq=0.  
\end{align}

Fix $d\in\cD(T)$ with $d\ge1$. Let $1\le k\le d$ and define $T_{d,k}$ to be
the tree that has a root of degree $d$, the first $k$ children of the root  
are copies of $T$, and the remaining $d-k$ children of the root are leaves.
Thus $|T_{d,k}|=1+k|T|+d-k$ and
\begin{align}\label{tx3}
  n_{T_{d,k}}(i)=k n_T(i) + (d-k)\indic{\{i=0\}}+\indic{\{i=d\}}, \hspace*{4mm} i \geq 0.
\end{align}
Moreover, there are exactly $k$ fringe trees in $T_{d,k}$ that are equal to $T$, so
$N_T(T_{d,k})=k$, while $N_{T_{d,k}}(T)=0$.
Hence, \eqref{tx2}, \eqref{eq21}, and \eqref{tx3} yield
\begin{align}\label{tx4}
  0 
&= \frac{\gammap(T,T_{d,k})}{\pip(T)\pip(T_{d,k})}
=\frac{k}{\pip(T)}+\etap(T,T_{d,k})
\notag\\&
=\frac{k}{\pip(T)}
+(|T|-1)(k|T|+d-k)-\sum_{i\in\cD(T)}
 \frac{n_T(i)\bigpar{k n_T(i) + (d-k)\indic{\{i=0\}}+\indic{\{i=d\}}}}{p_i}
\notag\\&
=\frac{k}{\pip(T)}
+(|T|-1)(k|T|+d-k)-k\sum_{i\in\cD(T)} \frac{n_T(i)^2}{p_i}
-(d-k)\frac{n_T(0)}{p_0}-\frac{n_T(d)}{p_d}.
\end{align}
Subtracting $k$ times \eqref{tx1} from \eqref{tx4} yields
\begin{align}\label{tx5}
  0 = d(|T|-1)-(d-k)\frac{n_T(0)}{p_0}-\frac{n_T(d)}{p_d}.
\end{align}
This has to hold for every $k=1,\dots,d$.
If $d\ge2$, this is a contradiction, since $n_T(0)>0$ for every tree $T$.

It remains only to consider the case when $\cD(T)$ has no element $d$ with
$d\ge2$, i.e., when $\cD(T)\subseteq\set{0,1}$.
In this case, $|T|$ is a path, so $n_{T}(0)=1$, $n_{T}(1)=|T|-1$ and $n_{T}(i) = 0$ for $i \geq 2$. We still have \eqref{tx5} with $k=d=1$, which gives
\begin{align}\label{tx6}
|T|-1 = \frac{n_T(1)}{p_1}  
=\frac{|T|-1}{p_1}.
\end{align}
Hence, $p_1=1$. Since $\sum_{i \geq 0}p_i=1$, this implies $p_0=0$, which contradicts
$\pip(T)>0$ because $n_{T}(0)>0$. 

These contradictions complete the proof.
\end{proof}

\begin{proof}[Proof of \refT{TheoCLT2}]
The proof is very similar to the proof of \refT{TheoCLT}, and we 
focus on the differences.
We may again assume that $T_1,\dots,T_m$ are distinct,
and we define again
$\mu_{i \kappa} \coloneqq \mu_\bnk(T_i)=|\mathbf{n}_{\kappa}|
\pi_{\mathbf{p}(\mathbf{n}_{\kappa})} (T_{i})$ for $1\le i\le m$.
Furthermore,
by \eqref{ta1}, $\E N_{T_j}(\Tnk)=\mu_\bnk(T_j)+O(1)$ 
for $1 \leq j \leq m$, and thus it does not matter whether we use
$\E N_{T_j}(\Tnk)$ or $\mu_{i\kk}=\mu_\bnk(T_j)$ in \eqref{eq77}. 

We now assume that $q_{1\kappa}, \dots, q_{m \kappa} \in \mathbb{N}_{0}$
are such that for some fixed constant $C >0$,
\begin{align}\label{gw5}
q_{i \kappa} \le C \mu_{i\kk}\qq, \qquad \text{for} \hspace*{2mm} 1 \leq i \leq m
.\end{align}
By \eqref{ta3}, we then have
\begin{align}
  \label{gw7}
q_{j\kk}=O\bigpar{\mu_{j\kk}\qq}=O\bigpar{n_\kk(i)\qq},
\qquad \text{for} \hspace*{2mm} 
i\in\cD(T_j), 
\end{align}
and in particular, $q_{j\kk}=O\bigpar{|\bnk|\qq}$.
Furthermore, by assumption $\mu_{j\kk}\to\infty$, and thus \eqref{ta3}
yields
$n_\kk(i)\to\infty$ for every $i\in\bigcup_{j=1}^{m}\cD(T_j)$.
Recall also the definition of $\tau_{ij}$ in (\ref{ea3}), for $1 \leq i,j \leq
m$. For $1 \leq j,k \leq m$, 
note that if $\tau_{jk}>0$, then $T_j$ is a fringe subtree of $T_k$; hence
$n_{T_j}(i) \le n_{T_k}(i)$ for every $i \geq 0$, and
thus, by \eqref{mu},
$\mu_{k\kk}\le\mu_{j\kk}$.
Hence, by \eqref{gw5},
if $\tau_{jk}>0$, then
$q_{k\kk} = O\bigpar{\mu_{k\kk}\qq} = O\bigpar{\mu_{j\kk}\qq}$.
Consequently, for every $1 \leq j\le m$,
\begin{align}\label{gw6}
 \sum_{k=1}^m\tau_{jk}q_{k\kk}  = O\bigpar{\mu_{j\kk}\qq}.
\end{align}

We now argue as in the proof of \refT{TheoCLT}.
It is easily checked that
all calculations in \eqref{er1}--\eqref{gw8} are valid in the present
situation too, using \eqref{gw5}--\eqref{gw6}
instead of \eqref{ju2}--\eqref{ju4}.

We then use \eqref{gw8} together with
\eqref{gw3} and \eqref{mu} and obtain
\begin{align}\label{gw9}
  \tr_\kk(j,k)|\bnk|
=\hgamma_\bpnk(T_j,T_k)
  \bigpar{\mu_\bnk(T_j)\mu_\bnk(T_k)}\qq-\gd_{jk}\mu_\bnk(T_j)
=\hgamma_\bpnk(T_j,T_k) \xpar{\mu_{j\kk}\mu_{k\kk}}\qq-\gd_{jk}\mu_{j\kk}
.\end{align}
As $\kappa \rightarrow \infty$, 
we have $\hgamma_\bpnk(T_j,T_k)\to\hgamma_\bp(T_j,T_k)$
by Condition \ref{Condition1}
and the continuity of $\hgamma_\bp(T_j,T_k)$ in $\bp$.
Hence, if we now define $\gs_{j\kk}:=\mu_{j\kk}\qq$, for $1 \leq j \leq m$,
then \eqref{eb2} and 
\eqref{gw9} yield
\begin{align}\label{gw10}
  \E\bigsqpar{(N_{T_{1}}(\mathcal{T}_{\mathbf{n}}))_{q_{1 \kappa}}
\dotsm (N_{T_{m}}(\mathcal{T}_{\mathbf{n}}))_{q_{m  \kappa}}}
=\prod_{j=1}^{m} \mu_{j\kappa}^{q_{j \kappa}}
 \exp \left( \frac{1}{2} \sum_{j,k=1}^{m} 
\frac{\xpar{\hgamma_\bp(T_j,T_k)\gs_{j\kk}\gs_{k\kk}-\gd_{jk}\mu_{j\kk}} }{\mu_{j\kappa} \mu_{k\kappa} }q_{j \kappa}q_{k \kappa}+o(1) 
\right),
\end{align}
uniformly in all $q_{1\kk},\dots,q_{m\kk}$ that satisfy \eqref{gw5}.
Since $\mu_{j\kk}/\gs_{j\kk}=\mu_{j\kk}\qq$ for $1 \leq j \leq m$, this is precisely the condition
\eqref{tgw2} in \refT{TGW}.
Moreover, $\mu_{j\kk}=\gs_{j\kk}^2$ and $\mu_{j\kk}\to\infty$ by assumption;
thus \eqref{tgw1} holds too.
Hence, \refT{TGW} applies and yields \eqref{eq77}.

The final claim follows by \refL{L+2} below. 
\end{proof}

\begin{lemma}\label{L+2}
Suppose that $\bp=(p_i)_{i \geq 0}\in\ppi$, and 
let\/ $T\in\bbT$. Then $\hgamma_\bp(T,T)>0$,
except in the three exceptional cases \ref{exc1}, \ref{exc2} and \ref{exc3}
of \refT{TheoCLT2}.
\end{lemma}

\begin{proof}
  If $\pip(T)>0$, then the result follows by \refL{L+1}.
Thus suppose $\pip(T)=0$.
Then, by \eqref{gw2},
\begin{align}
  \label{uj1}
\heta_\bp(T,T)=-\sum_{i\in\cD(T)}n_T(i)^2\prod_{j\in\cD(T)} p_j^{n_T(j)-\gd_{ij}}.
\end{align}
Since $\pip(T)=0$, there exists at least one $i_0\in\cD(T)$ with $p_{i_0}=0$.
Fix one such $i_0$.
Then each product in \eqref{uj1} with $i\neq i_0$ vanishes because it
contains the factor $p_{i_0}^{n_T(i_0)}=0$.
Hence, 
\begin{align}
  \label{uj2}
\heta_\bp(T,T)=-n_T(i_0)^2\prod_{j\in\cD(T)} p_j^{n_T(j)-\gd_{i_0j}}
.\end{align}
If $n_T(i_0)\ge2$, then \eqref{uj2} yields $\heta_\bp(T,T)=0$, and thus
$\hgamma_\bp(T,T)=1$  by \eqref{gw20}.
In the remaining case, $n_T(i_0)=1$; thus \eqref{gw20}  and \eqref{uj2} yield
\begin{align}
  \label{uj3}
\hgamma_\bp(T,T)=1+\heta_\bp(T,T)=1-\prod_{j\in\cD(T), j\neq i_0} p_j^{n_T(j)}
.\end{align}
Consequently, if $\hgamma_\bp(T,T)=0$, then $p_j=1$ for every $j\in\cD(T)$ with
$j\neq i_0$.
Obviously, there is at most one such $j$, and thus $|\cD(T)|\le2$.

We always have $0\in\cD(T)$, and thus either $|T|=1$ (case \ref{exc1}),
or
$\cD(T)=\set{0,d}$ for some $d\ge1$.
In the latter case, we have either $i_0=0$ or $i_0=d$. 

If $i_0=0$, then, as shown above, $n_{T}(0)=1$, so $T$ has only one leaf and
thus $T$ is a path. Then $\cD(T)=\set{0,1}$ and we need $p_1=1$; this is
case \ref{exc2}. 

If $i_0=d$, then  $n_{T}(d)=1$.
Thus, $T$ has only one non-leaf, so $T$ is a star where the root has degree
$d$; furthermore, $p_0=1$.
This is \ref{exc3}.
\end{proof}

While the exceptional case \ref{exc1} in \refT{TheoCLT2} (and \refL{L+2})
is completely trivial,  with $N_T(\cT_{\bnk})$ deterministic, 
the cases \ref{exc2} and \ref{exc3} are not.
We illustrate this with a simple example, which shows that in some such
cases $N_T(\cT_\bnk)$ is still asymptotically normal, but not in all cases.
                                                                 
\begin{example}\label{Eexc}
Let $T$ be the tree with $|T|=2$; thus $T$ consists of a root and a leaf,
and $n_T(0)=n_T(1)=1$.
Note that $T$ is an example of both exceptional cases \ref{exc2} (if $p_1=1$)
and \ref{exc3} (if $p_0=1$).

We consider for simplicity only degree statistics $\bnk$ such that $\Tnk$
has exactly one vertex of degree $\ge2$; it then follows from \eqref{eq30} that
this degree equals $n_\kk(0)$, and thus the degree statistic $\bnk$ has
$n_\kk(0)\ge2$, $n_\kk(1)\ge0$, and $n_\kk(i)=\gd_{i,n_\kk(0)}$, for $i\ge2$.
For such $\bnk$, the tree $\Tnk$ consists of a vertex, $v$ say, of
degree $n_\kk(0)$, $n_\kk(0)$ paths from $v$ to the leaves, and
a path (which might be empty) from the root to $v$.
Let $X_0\ge0$ be the number of vertices on the path from the root to $v$
(thus $X_0=0$ if $v$ is the root),
and let $X_i\ge0$ be the number of vertices of degree 1 on the $i$-th path from
$v$ to a leaf.
Then,
\begin{align}\label{kk1}
\sum_{i=0}^{n_\kk(0)}X_i = \nx(1),
\end{align}
 and there is a bijection between 
such vectors  $(X_i)_{i=0}^{\nx(0)}\in\bbNo^{\nx(0)+1}$ 
and possible trees $\Tnk$.
Vectors  $(X_i)_{i=0}^{\nx(0)}\in\bbNo^{\nx(0)+1}$ satisfying \eqref{kk1}
are called compositions  of $\nx(1)$ in $\nx(0)+1$ parts.
Hence, the random tree $\Tnk$ 
corresponds to a uniformly random composition  of $\nx(1)$ in $\nx(0)+1$ parts.
As is well known, 
see \cite[Example 1.3.3]{Kolchin1986}, \cite{Holst1979}, 
and e.g.\ \cite[Example 12.2]{Svante2012}, such a random composition can be
obtained by choosing any $p\in(0,1)$ and then letting $X_i\sim\Ge(p)$ be
independent, and condition on the event that \eqref{kk1} holds.

There is one fringe subtree $T$ in each path to a leaf
for which $X_i\ge1$; thus
we obtain, with $x_+:=\max(x,0)$,
\begin{align}\label{kk2}
  N_T(\Tnk)\eqd
\lrpar{\sum_{i=1}^{\nx(0)}\indic{X_i\ge1} \biggm| \sum_{i=0}^{\nx(0)}  X_i=\nx(1)}
= \nx(1)-\lrpar{\sum_{i=0}^{\nx(0)}(X_i-1)_+ +\indic{\{X_0\ge1\}}
\biggm| \sum_{i=0}^{\nx(0)}  X_i=\nx(1)}.
\end{align}

Consider now the case when $\nx(0)\to\infty$ and $\nx(1)=o(\nx(0))$, as $\kappa \rightarrow \infty$.
This implies that \refCond{Condition1} holds, with $p_0=1$.
By symmetry, $\E\bigsqpar{X_0\mid \sum_{i=0}^{\nx(0)}  X_i=\nx(1)}
=\nx(1)/(\nx(0)+1)\to0$, and thus we may ignore the term $\indic{\{X_0\ge1\}}$
in \eqref{kk2}.

For example, suppose that $\nx(1)/\sqrt{\nx(0)}\to\gl\in(0,\infty)$.
It is then easy to see that
$\nx(1)-N_T(\Tnk)\dto\Po(\gl^2)$ with a Poisson
limit distribution; see \cite[Example 5]{Holst1979}.
Moreover, by the methods in \cite{Holst1979}, see also
\cite[Theorem 2.1]{Janson2001},
it follows easily that
if
$\nx(0)\qq\ll\nx(1)\ll\nx(0)$, then $N_T(\Tnk)$ is asymptotically normal;
the variance is $\sim \nx(1)^2/\nx(0)$.
Conversely, it is  easy to see that if $\nx(1)\ll\nx(0)\qq$,
then $\P(N_T(\Tnk)=\nx(1))\to1$, so the distribution is asymptotically
degenerate. 

It is interesting to note that if $\nx(0)^{3/4}\ll\nx(1)\ll\nx(0)$,
then the asymptotic normality of $N_T(\Tnk)$ can easily be proved by the
Gao--Wormald theorem,
similarly to \refT{TheoCLT2}, using Lemmas \ref{LemC}\ref{LemC1} and
\ref{L0}. 
(With
 $\mu_{\kk}=|\nx|p_0(\nx)p_1(\nx)$ and
$\gs_{\kk}=p_1(\nx)|\nx|^{1/2}=\nx(1)/|\nx|\qq$.) 
However, we do not see how to use this method to prove the full range of
asymptotic normality in this example.

We have here concentrated on case \ref{exc3}, i.e., $p_0=1$.
By similar arguments, one can also study the case $\nx(1)/\nx(0)\to\infty$, 
when \refCond{Condition1} holds with $p_1=1$, and we are in the exceptional case
\ref{exc2}. Again, normal, Poisson and degenerate limits occur for various
ranges;
we leave the details to the reader.
\end{example}

\refE{Eexc} treats only a simple example, and we leave the general case
 as an open problem.

\begin{problem}
Find criteria for asymptotic normality of $N_T(\Tnk)$ in the exceptional
cases \ref{exc2} and \ref{exc3} in \refT{TheoCLT2}.
When is there a Poisson limit?
Are there  any other possible non-degenerate limit distributions?
\end{problem}

\section{Application to labelled trees with given vertex degrees}
\label{Slabelled}

For $n \in \mathbb{N}$, let $\mathbb{T}^{\rm lab}_{n}$ be the set of
unordered rooted trees with $n$ vertices labelled by $\{1, \dots, n\}$.
(I.e., the labelled rooted trees of size $n$.)
We use the notations above for such trees too, \emph{mutatis mutandis}. 
In particular,  for a tree $T\in \mathbb{T}^{\rm lab}_{n}$ and a vertex $i\in T$, 
$d_{T}(i)$ is the
(out)degree of $i \in T$. 
We define the degree sequence of $T$ 
as the sequence $\bd_T = (d_{T}(i))_{i=1}^n$.

Let
$\bbD_n:=\set{\bd_T:T\in\bbTlab_n}$ 
be the set
of degree sequences of labelled trees of size $n$;
if $\bd\in\bbD_n$, we say that $\bd$ is a
\emph{degree sequence of length $n$}.
Note that 
$\bbD_n:=\bigset{\bd=(d_i)_{i=1}^n\in\bbNo^n:\sum_{i=1}^n d_i=n-1}$.
We further let $\bbD:=\bigcup_{n\ge1}\bbD_n$, the set of all degree
sequences.
If $\bd$ is a degree sequence, we write $|\bd|=n$ if $\bd\in\bbD_n$; 
we then say that $n$ is the \emph{length} of $\bd$.
We also define the degree statistic $\bn_\bd=(n_\bd(i))_{i\ge0}$,
where $n_\bd(i):=|\set{v\in\set{1,\dots,n}:d_v=i}|$.
Note that $|\bn_\bd|=|\bd|$.

For a degree sequence $\bd$, let 
$\bbTlab_{\bd}$ be the set of labelled trees $T\in\bbTlab_{|\bd|}$ that have
degree sequence $\bd$.
We let $\cTlab_\bd$ be a uniformly random element of $\bbTlab_\bd$, i.e., a
uniformly random labelled tree with degree sequence $\bd$; we denote this by
$\cTlab_\bd \sim {\rm Unif}(\bbTlab_\bd)$.

Although the random trees $\cT_\bn$ (for a degree statistics $\bn$)
and $\cTlab_\bd$ (for a degree sequence $\bd$)
are different types of
trees, it is well known that they are closely related and for many purposes
equivalent. We state one version of this as a lemma.

\begin{lemma}\label{Llab}
Let\/ $\bd$ be a degree sequence, let\/ $\bn_\bd$ be the corresponding degree
statistic, and let\/
$\mathcal{T}_{\bn_\bd} \sim {\rm Unif}(\mathbb{T}_{\bn_\bd})$.
We may construct\/ $\cTlab_\bd \sim {\rm Unif}(\bbTlab_\bd)$
as follows: 
randomly label the vertices of\/ $\cT_{\bn_\bd}$ such that the tree has degree
sequence $\bd$, and then ignore the ordering.
\end{lemma}

\begin{proof}
Let $\ocTlabd$ be the intermediary labelled ordered random tree.
The tree $\cT_{\bn_\bd}$ may be labelled in exactly  $\prod_{i\ge0}n_\bd(i)!$
ways to have the specified degree sequence $\bd$. Since this number is
constant (for a given $\bd$), $\ocTlabd$ is uniformly distributed over all
labelled ordered trees with degree sequence $\bd$.
Similarly, each labelled unordered tree with degree sequence
$\bd=(d_i)_{i=1}^n$ can be 
ordered in $\prod_{i=1}^n d_i!$  ways; again this number is constant, and
thus the tree obtained from $\ocTlabd$ by forgetting the ordering 
is uniformly distributed on $\bbTlab_\bd$.
\end{proof}

For a tree $T\in\bbTlab_n$ and a vertex $v\in T$, we define the fringe tree
$T_v$ as before. We ignore the labels on $T_v$; thus, $T_v$ is regarded as
an unordered unlabelled rooted tree.
Let $\bbTun$ be the set of unordered unlabelled rooted trees.
If $T\in\bbTlab_n$ and $T'\in\bbTun$, let as before $N_{T'}(T)$ 
be
the number of fringe trees of $T$ that are equal (i.e., isomorphic to) $T'$;
this is again given by \eqref{eq2}.

For a tree $T\in\bbTun$, let $\Ord(T)$ be the set of 
ordered trees $\Tord\in\bbT$ that reduce to $T$ if we ignore the ordering.
It follows from the construction in \refL{Llab} that for any degree sequence
$\bd$ and tree $T\in\bbTun$, 
\begin{align}\label{lab1}
  N_{T}(\cTlab_\bd)=\sum_{\Tord\in\Ord(T)}N_{\Tord}(\cT_{\bn_\bd}).
\end{align}
Versions for random labelled trees $\bbTlab_{\bd}$ 
of Theorems \ref{TheoLLN}, \ref{TheoCLT}, \ref{TA} and \ref{TheoCLT2} now
follow as a consequence of \eqref{lab1}.
We state only the two first of these in detail, and leave the others to the
reader. We first need some notation.

Note that the definitions in \refS{S:Intro} of 
$\pi_\bp(T)$,
$\eta_\bp(T,T')$,
$\gamma_\bp(T,T')$,
$\heta_\bp(T,T')$, and
$\hgamma_\bp(T,T')$
use only the degree statistics and not the orderings;
these quantities are thus well defined also for unordered trees
$T,T'\in\bbTun$;
moreover, they have the same value as if we give the trees any orderings.
Recall also that $\cTp$ is a Galton--Watson tree with offspring distribution
$\bp$; we let $\cTun_\bp$ denote this Galton--Watson tree regarded as an
unordered tree in $\bbTun$.
In analogy with \eqref{pip} we define, for $T\in\bbTun$,
\begin{align}\label{lab11}
  \piun_\bp(T):=\P\bigpar{\cTun_\bp=T}=\ord{T}\pi_\bp(T).
\end{align}
Furthermore, in analogy with \eqref{eq20}--\eqref{eq21},
for $T,T'\in\bbTun$,
\begin{align} 
\gamma\un_{\mathbf{p}}(T, T) &  
:= \piun_{\mathbf{p}}(T) + \eta_{\mathbf{p}}(T,T)(\pi\un_{\mathbf{p}}(T))^{2}, 
\label{lab20}\\
\gamma\un_{\mathbf{p}}(T, T') & 
:= N_{T'}(T)  \piun_{\mathbf{p}}(T) +  N_{T}(T') \piun_{\mathbf{p}}(T') 
+\eta_{\mathbf{p}}(T,T') \pi\un_{\mathbf{p}}(T) \pi\un_{\mathbf{p}}(T'), 
\qquad T\neq T'. 
\label{lab21}
\end{align}

\begin{theorem} \label{TLLNlab}
  Let $\bd_\kk$, $\kk\ge1$, be some degree sequences such that the
  corresponding degree statistics $\bn_{\bd_\kk}$ satisfy
  \refCond{Condition1}, and let
$\cTlab_{\bd_\kk} \sim {\rm Unif}(\bbTlab_{\bd_\kk})$.
For every fixed $T \in \bbTun$,
as \kktoo:
\begin{enumerate}[label=\upshape(\roman*)]
\item (Annealed version)\quad 
$\displaystyle \mathbb{P}(\cTlabran_{\bd_\kk} = T) 
=
\frac{\E[N_{T}(\cTlab_{\bd_\kk})]}{|\bd_\kk|}
\to \piun_{\mathbf{p}}(T)$. 
\label{AnnealedLab}

\item (Quenched version)\quad 
$\displaystyle \mathbb{P}(\cTlabran_{\bd_\kk} 
= T\mid \cTlab_{\bd_\kk} ) =
  \frac{N_{T}(\cTlab_{\bd_\kk})}{|\bd_\kk|}
  \to \piun_{\mathbf{p}}(T)$ 
in probability. 
\label{QuenchedLab}
\end{enumerate}
In other words, 
the random fringe tree converges in distribution 
as $\kk\to\infty$:
\ref{AnnealedLab} says
$\cTlabran_{\bd_\kk}\dto \cTun_\bp$,
or equivalently $\cL(\cTlabran_{\bnk})\to \cL(\cTun_\bp)$, 
and \ref{QuenchedLab} is the conditional version
$\cL\bigpar{\cTlabran_{\bnk}\mid \cTlab_{\bdk}}\pto \cL(\cTun_\bp)$. 
\end{theorem}
\begin{proof}
This follows by \refT{TheoLLN} together with \eqref{lab1} since,
for any $T\in\bbTun$, using \eqref{lab11},
\begin{align}\label{lab2}
\sum_{\bT\in\Ord(T)}\pip(\bT)    
&=
\ord{T}\pip(T) 
=\piun(T).
\end{align}
\end{proof}

\begin{theorem}\label{TCLTlab}
  Let $\bd_\kk$, $\kk\ge1$, be some degree sequences such that the
  corresponding degree statistics $\bn_{\bd_\kk}$ satisfy
  \refCond{Condition1}, and let
$\cTlab_{\bd_\kk} \sim {\rm Unif}(\bbTlab_{\bd_\kk})$.
For a fixed $m \geq 1$, let
$T_{1}, \dots, T_{m} \in \bbTun$ be a fixed sequence of
rooted unordered unlabelled trees.
Then, as $\kappa \rightarrow \infty$,
\begin{align} 
\E N_{T_{i}}(\cTlab_{\bd_\kk})  & 
=  \piun_{\mathbf{p}}(T_{i})|\bd_\kk|  + o(|\bd_\kk|), \label{eq6lab}\\
\Var N_{T_{i}}(\cTlab_{\bd_\kk}) & 
=   \gamma\un_\bp(T_i,T_{i})|\bd_\kk| + o(|\bd_\kk|), \label{eq15lab} \\
{\rm Cov}\bigpar{N_{T_{i}}(\cTlab_{\bd_\kk}),
  N_{T_{j}}(\cTlab_{\bd_\kk})} & 
=   \gamma\un_\bp(T_i,T_j)|\bd_\kk| + o(|\bd_\kk|), 
\label{eq13lab}
\end{align}
for $1 \leq i,j \leq m$,
and
\begin{align} \label{eq7lab}
\left(\frac{ N_{T_{1}}(\cTlab_{\bd_\kk}) -
  \E[N_{T_{1}}(\cTlab_{\bd_\kk})]}{\sqrt{|\bd_\kk|}}, 
\dots, 
\frac{ N_{T_{m}}(\cTlab_{\bd_\kk}) -
  \E[N_{T_{m}}(\cTlab_{\bd_\kk})]}{\sqrt{|\bd_\kk|}}
  \right) 
\dto\N(0, \Gamma\un_{\mathbf{p}}),
\end{align}
\noindent where the covariance matrix 
$\Gamma\un_{\mathbf{p}} :=(\gamma\un_{\mathbf{p}}(T_{i},T_{j}))_{i,j=1}^m$. 
Furthermore, in \eqref{eq7lab}, 
we can replace
$\E[N_{T_{i}}(\cTlab_{\bd_\kk})]$ 
by $|\bd_\kk|\,\piun_{\mathbf{p}(\mathbf{n}_{\bd_\kk})}(T_{i})$. 
\end{theorem}

\begin{proof}
This follows by \refT{TheoCLT} together with \eqref{lab1}, using
\eqref{lab2}  and the following calculations.
First, for any $T\in\bbTun$,
by \eqref{eq20}--\eqref{eq21},
\eqref{lab11}, 
and the fact that if $\bT,\bT'\in\bbT$ with $|\bT|=|\bT'|$ but
$\bT\neq\bT'$, then $N_{\bT}(\bT')=0$,
\begin{align}\label{lab3}
\sum_{\bT\in\Ord(T),\,\bT'\in\Ord(T)} \gammap(\bT,\bT')&
=
\sum_{\bT\in\Ord(T),\,\bT'\in\Ord(T)} 
\bigpar{\pip(T)\indic{\{\bT=\bT'\}}+\etap(\bT,\bT')\pip(T)^2}
\notag\\&
=\ord{T}\pip(T)+\ord{T}^2\pip(T)^2\etap(T,T)
=\gamma\un_\bp(T,T).
\end{align}
Secondly, 
for $T,T'\in\bbTun$ with $T\neq T'$, we have,
cf.\ \eqref{lab1},
\begin{align}\label{lab5}
\sum_{\bT\in\Ord(T),\,\bT'\in\Ord(T')} N_{\bT'}(\bT)
=\sum_{\bT\in\Ord(T)} N_{T'}(\bT)
=\ord{T} N_{T'}(T)
,\end{align}
and thus, similarly,
by \eqref{eq21}, \eqref{lab11} and \eqref{lab21},
\begin{align}\label{lab4}
\sum_{\bT\in\Ord(T),\,\bT'\in\Ord(T')} \gammap(\bT,\bT')&
=
\sum_{\bT\in\Ord(T),\bT'\in\Ord(T')} 
\bigpar{N_{\bT'}(\bT)\pip(T)+N_{\bT}(\bT')\pip(T')
+\etap(T,T')\pip(T)\pip(T')}
\notag\\&
=\gamma\un_\bp(T,T')
.\end{align}
\end{proof}

  \begin{problem}
Suppose that $T\in\bbTun$ and $\bp\in\ppi$ with $|T|>1$ and $\pip\un(T)>0$.
Is $\gamma\un_\bp(T,T)>0$?
  \end{problem}
Note that an affirmative answer to Problem \ref{ProbGamma} would imply 
a positive answer to this too.

\section{Application to simply generated trees} 
\label{SimplyTrees}

Let $\mathbb{T}_{n}$ denote the (finite) subset of all plane rooted trees of size $n \in \mathbb{N}$. Let $\mathbf{w} = (w_{i})_{i \geq 0}$ be a sequence of non-negative real weights with $w_{0}>0$ and $w_{i} >0$ for at least one $i \geq 2$. For a finite rooted plane tree $T \in \mathbb{T}$, we define the weight of $T$ to be
\begin{align}\label{fi1}
w(T) \coloneqq \prod_{v \in T} w_{d_{T}(v)}
=\prod_{i\ge0}w_i^{n_T(i)}. 
\end{align}
\noindent For $n \in \mathbb{N}$, let $Z_{n}(\mathbf{w}) = \sum_{T \in \mathbb{T}_{n}}w(T)$. If $Z_{n}(\mathbf{w}) >0$, then we define the random tree $\mathcal{T}_{\mathbf{w},n}$ by picking an element of $\mathbb{T}_{n}$ at random with probability proportional to its weight, i.e., 
\begin{align} \label{eq23}
\mathbb{P}(\mathcal{T}_{\mathbf{w},n} = T) = \frac{w(T)}{Z_{n}(\mathbf{w})}, \hspace*{4mm} \text{for} \hspace*{2mm} T \in \mathbb{T}_{n}. 
\end{align}
\noindent The random tree $\mathcal{T}_{\mathbf{w},n}$ is called simply
generated tree of size $n$ and weight sequence $\mathbf{w}$; see
e.g.\ \cite{Drmota2009} and  \cite{Svante2012}. 
If $\mathbf{w}$ is a probability distribution (i.e., $\sum_{i \geq 0} w_{i} = 1$), then $\mathcal{T}_{\mathbf{w},n}$ is a Galton--Watson tree with offspring distribution $\mathbf{w}$ conditioned to have $n$ vertices. 

Let $\Phi_{\mathbf{w}}(z) = \sum_{i \geq 0} w_{i} z^{i}$ be the generating
  function of the weight sequence $\mathbf{w}$, and let $\rho_{\mathbf{w}}
  \in [0, \infty]$ be its radius of convergence. 
For $0 \leq s < \rho_{\mathbf{w}}$, we let
\begin{eqnarray}\label{fi2}
\Psi_{\mathbf{w}}(s) := \frac{s \Phi_{\mathbf{w}}^{\prime}(s) }{\Phi_{\mathbf{w}}(s) } = \frac{\sum_{i \geq 0} i w_{i}s^{i}}{\sum_{i \geq 0} w_{i}s^{i}}. 
\end{eqnarray}
Furthermore, if $\Phi_{\mathbf{w}}(\rho_{\mathbf{w}}) < \infty$, 
we define also $\Psi_{\mathbf{w}}(\rho_{\mathbf{w}})$ by \eqref{fi2};
if
  $\Phi_{\mathbf{w}}(\rho_{\mathbf{w}}) = \infty$ then we define
$\Psi_{\mathbf{w}}(\rho_{\mathbf{w}}) 
:= \lim_{s \uparrow \rho_{\mathbf{w}}} \Psi_{\mathbf{w}}(s)$; the limit
exists by \cite[Lemma  3.1 (i)]{Svante2012}. 
Let $\nu_{\mathbf{w}} :=
  \Psi_{\mathbf{w}}(\rho_{\mathbf{w}}) \in [0, \infty]$, and define 
\begin{align}\label{fi3} 
\tau_{\mathbf{w}} = \begin{cases}
              \rho_{\mathbf{w}} & \text{if }  \nu_{\mathbf{w}} < 1,\\
             \Psi^{-1}_{\mathbf{w}}(1)  & \text{if }  \nu_{\mathbf{w}}\ge1. \\
              \end{cases}
\end{align}
It follows from \cite[Lemma 3.1]{Svante2012} that
\begin{align}\label{sf0}
\rho_\bw>0\iff\nu_\bw>0\iff \tau_\bw>0.  
\end{align}

The following result from \cite{Svante2012} shows that simply generated trees
  satisfy Condition \ref{Condition1} in probability.  

\begin{theorem}[{\cite[Theorem 7.1 and Theorem 7.11]{Svante2012}}] 
\label{SvanteDegree}
Let $\mathbf{w}$ be a sequence of non-negative real weights with $w_{0}>0$ and $w_{i} >0$ for at least one $i \geq 2$. Define 
\begin{align}\label{sf1}
\theta_{i}(\mathbf{w}) = \frac{w_{i} \tau_{\mathbf{w}}^{i}}{\Phi_{\mathbf{w}}(\tau_{\mathbf{w}})}, \hspace*{4mm} \text{for} \hspace*{2mm} i \geq 0.
\end{align}
\noindent Then, $\theta(\mathbf{w}) = (\theta_{i}(\mathbf{w}))_{i \geq 0}$  is
a probability distribution with expectation $\mu_{\mathbf{w}} = \min(1,
\nu_{\mathbf{w}})$ and variance $\sigma^{2}_{\mathbf{w}} = \tau_{\mathbf{w}}
\Psi_{\mathbf{w}}^{\prime}(\tau_{\mathbf{w}}) \in [0, \infty]$. Moreover,
for $n \in \mathbb{N}$ with $Z_{n}(\mathbf{w}) >0$, let
$\mathcal{T}_{\mathbf{w},n}$ be a simply generated tree of size $n$ and
weight sequence $\mathbf{w}$. Then, the (empirical) degree distribution
$\mathbf{p}(\mathbf{n}_{\mathcal{T}_{\mathbf{w},n}})$ of
$\mathcal{T}_{\mathbf{w},n}$ satisfies, for every $i \geq 0$,
$p_{i}(\mathbf{n}_{\mathcal{T}_{\mathbf{w},n}}) \pto
\theta_{i}(\mathbf{w})$, as $n \rightarrow \infty$ (along integers $n$ such
that $Z_{n}(\mathbf{w}) >0$).
\end{theorem}

Note that if $\rho_{\mathbf{w}} = 0$, then $\theta_{0}(\mathbf{w})=1$ and
$\theta_{i}(\mathbf{w})=0$ for $i \geq 1$; otherwise, $\tau_\bw>0$ and 
\eqref{sf1} shows that $\gth_i(\bw)>0 \iff w_i>0$ for $i\ge0$.

Using \refT{SvanteDegree}, 
we will show that \refT{TheoCLT} implies the following version
for conditioned Galton--Watson trees.
The asymptotic normality \eqref{eq28} was proved in case \ref{case1GW}
by different methods in 
\cite[Corollary 1.8]{Svante2016};
\ref{case2GW} and \ref{case3GW} are new.

\begin{theorem}[partly \cite{Svante2016}] \label{TCorollary1}
Let $\mathbf{w}$ be a sequence of non-negative real weights with $w_{0}>0$
and $w_{i} >0$ for at least one $i \geq 2$. Moreover, for $n \in \mathbb{N}$
with $Z_{n}(\mathbf{w}) >0$, let $\mathcal{T}_{\mathbf{w},n}$ be a simply
generated tree of size $n$ and weight sequence $\mathbf{w}$. For fixed $m
\geq 1$, let $T_{1}, \dots, T_{m} \in \mathbb{T}$ be a fixed sequence of
rooted plane trees. 
Then, as $n \rightarrow \infty$ (along integers $n$ such that
$Z_{n}(\mathbf{w}) >0$), 
\begin{align} \label{eq22}
\left(\frac{ N_{T_{1}}(\mathcal{T}_{\mathbf{w},n}) -
  \E[N_{T_{1}}(\mathcal{T}_{\mathbf{w},n}) \mid
  \mathbf{n}_{\mathcal{T}_{\mathbf{w},n}}]}{\sqrt{n}}, \dots, \frac{
  N_{T_{m}}(\mathcal{T}_{\mathbf{w},n}) -
  \E[N_{T_{m}}(\mathcal{T}_{\mathbf{w},n}) \mid
  \mathbf{n}_{\mathcal{T}_{\mathbf{w},n}}]}{\sqrt{n}}\right) 
\dto {\rm N}(0, \Gamma_{\theta(\mathbf{w})}),  
\end{align}
\noindent where the covariance matrix $\Gamma_{\theta(\mathbf{w})}$ is
defined by \eqref{eq20}--\eqref{eq21}, and for $1 \leq j \leq m$,
\begin{align} \label{eq25}
\E[N_{T_{j}}(\mathcal{T}_{\mathbf{w},n}) \mid
  \mathbf{n}_{\mathcal{T}_{\mathbf{w},n}}] =  \frac{n}{(n)_{|T_{j}|}}
  \prod_{i \geq 0} (n_{\mathcal{T}_{\mathbf{w},n}}(i))_{n_{T_{j}}(i)}.
\end{align}
\noindent Furthermore, suppose that the weight sequence $\mathbf{w}$ satisfies one of the following conditions:
\begin{enumerate}[label=\upshape(\roman*)]
\item $\nu_{\mathbf{w}} \geq 1$ and $\sigma^{2}_{\mathbf{w}} \in (0,\infty)$. \label{case1GW}
\item $\nu_{\mathbf{w}} \geq 1$, $\sigma^{2}_{\mathbf{w}} = \infty$ and
  $\theta(\mathbf{w})$ belongs to the domain of attraction of a stable law
  of index $\alpha \in (1,2]$. 
(The last condition is equivalent to that
there exists a slowly varying function
  $L:\mathbb{R}_{+} \rightarrow \mathbb{R}_{+}$ such that $\sum_{i = 0}^{k}
  i^{2} \theta_{i}(\mathbf{w}) = k^{2-\alpha} L(k)$,
  as $k \rightarrow \infty$
\cite[Theorem XVII.5.2]{FellerII}.) \label{case2GW}

\item $0 < \nu_{\mathbf{w}} < 1$ and $\theta_{i}(\mathbf{w}) = c i^{- \beta} + o(i^{- \beta})$, as $i \rightarrow \infty$, with fixed $c >0$ and $\beta >2$. \label{case3GW}
\end{enumerate}
\noindent Then, as $n \rightarrow \infty$ (along integers $n$ such that $Z_{n}(\mathbf{w}) >0$), 
\begin{align} \label{eq28}
\left(\frac{ N_{T_{1}}(\mathcal{T}_{\mathbf{w},n}) -  n \pi_{\theta(\mathbf{w})}(T_{1})}{\sqrt{n}}, \dots, \frac{ N_{T_{m}}(\mathcal{T}_{\mathbf{w},n}) -  n \pi_{\theta(\mathbf{w})}(T_{m})}{\sqrt{n}}\right) \dto {\rm N}(0, \xGamma_{\theta(\mathbf{w})}),  
\end{align}
\noindent where the covariance matrix 
$\xGamma_{\theta{(\mathbf{w})}} = (\xgamma_{\theta{(\mathbf{w})}}(T_{i},
T_{j}))_{i,j=1}^{m}$ is given by,
for $T,T'\in\bbT$ such that $T\neq T'$, 
\begin{align}\label{fi4} 
\xgamma_{\theta{(\mathbf{w})}}(T, T) & = \pi_{\theta(\mathbf{w})}(T) - \left(2 |T|- 1 + \varsigma_{\mathbf{w}}^{-2} \right) (\pi_{\theta(\mathbf{w})}(T))^{2}, 
\\\label{fi5}
\xgamma_{\theta{(\mathbf{w})}}(T, T') & = N_{T'}(T)  \pi_{\theta{(\mathbf{w})}}(T) +  N_{T}(T') \pi_{\theta{(\mathbf{w})}}(T') - \left (|T|+|T'|-1+ \varsigma_{\mathbf{w}}^{-2} \right)\pi_{\theta{(\mathbf{w})}}(T) \pi_{\theta{(\mathbf{w})}}(T'),
\end{align}
with $\varsigma^{2}_{\mathbf{w}} = \sigma^{2}_{\mathbf{w}}$ in case \ref{case1GW},
and $\varsigma^{2}_{\mathbf{w}} = \infty$ in cases
\ref{case2GW} and \ref{case3GW}. 
\end{theorem}

\begin{remark}\label{RGWequiv}
Recall that for any weight sequence $\mathbf{w}$ and any constants $a,b >0$, 
the weight sequence $\hbw = (\hw_{i})_{i \geq 0}$ with $\hw_{i} := ab^{i}w_{i}$ is 
equivalent to $\mathbf{w}$, i.e., it satisfies that 
$\mathcal{T}_{\mathbf{w},n} \stackrel{\rm d}{=} \mathcal{T}_{\hbw,n}$, 
for all $n$ for which either (and thus both) of the random trees are defined; 
this is a consequence of (\ref{eq23}). In the setting of 
Theorem \ref{SvanteDegree}, 
if $\rho_{\mathbf{w}} > 0$, then the weight sequence $\mathbf{w}$ is
equivalent to the weight sequence 
$\theta(\mathbf{w}) = (\theta_{i}(\mathbf{w}), i \geq 0)$, which is a  
probability distribution with mean $\mu_{\mathbf{w}}=\min(1,\nu_{\bw})$; 
see further \cite[Section 7]{Svante2012}. Thus, if $\rho_{\mathbf{w}} > 0$ we can 
regard $\mathcal{T}_{\mathbf{w},n}$ as a Galton--Watson tree 
$\mathcal{T}_{\theta(\mathbf{w}),n}$ with offspring distribution
$\theta(\mathbf{w})$ conditioned to have $n$ vertices. 
This explains the appearance of $\theta(\bw)$ in \refT{TCorollary1},
and it shows that there is no real loss of generality to consider (as is
often done) only the case $\tau_\bw=1$ when $\gth(\bw)=\bw$.
Note that the
conditioned Galton--Watson tree $\mathcal{T}_{\theta(\mathbf{w}),n}$ 
is critical if $\nu_{\mathbf{w}} \geq 1$, 
and subcritical if $0 < \nu_{\mathbf{w}} < 1$. 
\end{remark}

We postpone the proof of \refT{TCorollary1}.
A central idea is to obtain the unconditional limit \eqref{eq28} by
combining the conditional limit \eqref{eq22} with a limit result for the
conditional expectations in \eqref{eq25}.
For this, we will use the following theorem on asymptotic normality of the 
degree statistics, which is proved in \cite{Svante2016} and \cite{Paul2020}.
To be more precise, 
case \ref{case1GW} is shown in \cite[Example 2.2]{Svante2016},
while  cases \ref{case2GW} and \ref{case3GW} are shown in 
\cite[Theorems 6.2 and 6.7]{Paul2020} 
(although the asymptotic (co)variances are not explicitly given
in \cite[Theorem 6.2]{Paul2020}). 
Moreover, the approach used in the present paper allows us
to give a different (and simpler) proof of 
\refT{TheoDegreeS}  in cases
\ref{case1GW} and \ref{case2GW}, using
the multidimensional version of the Gao--Wormald theorem (Theorem \ref{TGW});
we give this  proof 
in Appendix \ref{AppendixB}.

\begin{remark}
In case \ref{case1GW}, i.e.\ $\nu_{\mathbf{w}} \geq 1$ and
$\sigma^{2}_{\mathbf{w}} \in (0,\infty)$, 
the univariate version of Theorem \ref{TheoDegreeS} was first proved by
Kolchin \cite[Theorem 2.3.1]{Kolchin1986};
the multivariate result \eqref{eq28} was proved 
in general in \cite{Svante2016} as said above, and
earlier under extra moment assumptions
on $\theta{(\mathbf{w})}$ 
by Janson \cite{Janson2001} (assuming a third moment),
Minami \cite{Minami2005} and Drmota \cite[Section 3.2.1]{Drmota2009} 
(both assuming an exponential moment), using different proofs. 
\end{remark}

\begin{remark}
If $\nu_{\mathbf{w}} \geq 1$ and $\sigma^{2}_{\mathbf{w}} \in (0,\infty)$ 
as in \refT{TCorollary1}\ref{case1GW}, 
then the offspring distribution $\theta(\mathbf{w})$ 
is critical (i.e., has mean 1) with finite variance. 
This is the framework assumed in \cite{Svante2016}, and 
as said above, in this case, \eqref{eq28} is proved in
\cite[Corollary 1.8]{Svante2016}. 
Our proof uses \refT{TheoDegreeS}, which in this case also is a result from
\cite{Svante2016}. However, note that our proof is quite different; 
we use \refT{TheoDegreeS} to prove \eqref{eq28}, while \cite{Svante2016}
essentially does the opposite.
\end{remark}

\begin{theorem}[\cite{Svante2016} and \cite{Paul2020}]
 \label{TheoDegreeS}
Let $\mathbf{w}$ and $\mathcal{T}_{\mathbf{w},n}$ be as in \refT{TCorollary1},
and assume that one of the conditions \ref{case1GW}--\ref{case3GW} there holds.
Then, for any fixed $k\in\bbNo$,
as $n \rightarrow \infty$ (along integers $n$ such that $Z_{n}(\mathbf{w}) >0$),
\begin{align} \label{fi11}
\left(\frac{ n_{\mathcal{T}_{\mathbf{w},n}}(0) - n  \theta_{0}(\mathbf{w})}{\sqrt{n}}, \dots, \frac{ n_{\mathcal{T}_{\mathbf{w},n}}(k) - n  \theta_{k}(\mathbf{w})}{\sqrt{n}}
\right)\dto {\rm N}(0, \dGamma_{k}),  
\end{align}
\noindent where the covariance matrix $\dGamma_{k} \coloneqq (\dgamma(i,j))_{i,j=0}^{k}$ is given by 
\begin{align} \label{fi12}
\dgamma(i, i) & = \theta_{i}(\mathbf{w})(1- \theta_{i}(\mathbf{w})) - (i-1)^{2} \theta_{i}(\mathbf{w})^{2}/\varsigma^{2}_{\mathbf{w}},  \\
\dgamma(i, j) & = - \theta_{i}(\mathbf{w})\theta_{j}(\mathbf{w}) 
- (i-1)(j-1) \theta_{i}(\mathbf{w}) \theta_{j}(\mathbf{w})/\varsigma^{2}_{\mathbf{w}},
\qquad i\neq j, \label{fi13}
\end{align}
where $\varsigma_\bw^2$ is as in \refT{TCorollary1}. 
(In particular, 
$\varsigma_{\bw}^{2}=\infty$ in cases \ref{case2GW} and \ref{case3GW}; we then
interpret the final terms in \eqref{fi12} and \eqref{fi13} as $0$.)
\end{theorem}

\begin{proof}[Proof of \refT{TCorollary1}]
For any fixed degree statistic $\mathbf{n}$ with
$\mathbb{P}(\mathbf{n}_{\mathcal{T}_{\mathbf{w},n}} = \mathbf{n}) >0$,
\eqref{eq23} implies that conditionally given 
$\mathbf{n}_{\mathcal{T}_{\mathbf{w},n}} =
\mathbf{n}$, $\mathcal{T}_{\mathbf{w},n} \sim {\rm
  Unif}(\mathbb{T}_{\mathbf{n}})$; see e.g., \cite[Proposition
8]{Addario2022}. By the Skorohod coupling theorem \cite[Theorem
4.30]{Kallenberg2002}, we can and will assume that the convergence in
Theorem \ref{SvanteDegree} holds a.s.;
in other words, \refCond{Condition1} holds a.s.\ 
for the degree statistics $\mathbf{n}_{\mathcal{T}_{\mathbf{w},n}}$, 
with $\bp=\theta(\bw)$.
Moreover, e.g.\ by resampling $\cT_{\bw,n}$ conditioned on 
$\mathbf{n}_{\mathcal{T}_{\mathbf{w},n}}$, we may assume that 
also conditioned on the entire
sequence of degree statistics
$(\mathbf{n}_{\mathcal{T}_{\mathbf{w},n}})_{n=1}^\infty$,
the random trees $\cT_{\bw,n}$, $n\ge1$,  
have the (conditional) distributions  
$\Unif(\bbT_{\bn_{\mathcal{T}_{\mathbf{w},n}}})$.
It follows that we may apply Theorem \ref{TheoCLT}
conditioned on the 
sequence of degree statistics
$(\mathbf{n}_{\mathcal{T}_{\mathbf{w},n}})_{n=1}^\infty$; 
this shows that (\ref{eq22}) holds
conditioned on
$(\mathbf{n}_{\mathcal{T}_{\mathbf{w},n}})_{n=1}^\infty$. Then, (\ref{eq22})
also holds unconditionally by the dominated convergence theorem. 
Furthermore, (\ref{eq25}) follows from Lemma \ref{lemma1}.

In the rest of the proof we assume that
either \ref{case1GW}, \ref{case2GW} or \ref{case3GW} holds, and thus 
\refT{TheoDegreeS} applies.
We fix $k$ so large that $k\ge i$ for every $i\in\bigcup_{j=1}^m \cD(T_j)$.
(Recall \eqref{cD}.)
For convenience, we use again the Skorohod coupling theorem, 
and may thus assume that the limits 
$\bp(\mathbf{n}_{\mathcal{T}_{\mathbf{w},n}})\to\theta(\bw)$
in Theorem \ref{SvanteDegree} and (\ref{fi11}) 
in \refT{TheoDegreeS}  hold almost surely. 

Let $1\le j\le m$.
First, 
suppose that $\pi_{\theta(\mathbf{w})}(T_{j}) >0$, 
i.e., $\gth_i(\bw)>0$ for $i\in\cD(T_j)$.
Then, a.s., by (\ref{eq25}) and Lemma \ref{L0} and the assumed a.s.\
convergence in \eqref{fi11},
\begin{align}
\E[N_{T_{j}}(\mathcal{T}_{\mathbf{w},n}) \mid \mathbf{n}_{\mathcal{T}_{\mathbf{w},n}}] & =  n \prod_{i \in\cD(T_j)} \left(\frac{n_{\mathcal{T}_{\mathbf{w},n}}(i)}{n} \right)^{n_{T_{j}}(i)} + O(1) \nonumber \\
& = n \exp \left(\sum_{i \in\cD(T_j)} n_{T_{j}}(i) \ln \left( \frac{n_{\mathcal{T}_{\mathbf{w},n}}(i) - n \theta_{i}(\mathbf{w})}{n} + \theta_{i}(\mathbf{w}) \right) \right) + O(1) \nonumber  \\
& = \pi_{\theta(\mathbf{w})}(T_{j}) n + \pi_{\theta(\mathbf{w})}(T_{j}) \sum_{i \geq 0} \frac{n_{T_{j}}(i)}{\theta_{i}(\mathbf{w})} \left(n_{\mathcal{T}_{\mathbf{w},n}}(i) - n \theta_{i}(\mathbf{w}) \right) + O(1). \label{eq26}
\end{align}
\noindent On the other hand, if $\pi_{\theta(\mathbf{w})}(T_{j}) =0$, then
$\theta_{i^{\prime}}(\mathbf{w}) = 0$ for some $i^{\prime} \geq 0$ such that
$n_{T_{j}}(i^{\prime}) >0$. 
Since we assume $\nu_\bw>0$, this implies by 
\eqref{sf0} and 
\eqref{sf1} that $w_{i'}=0$;
hence, $\cT_{\bw,n}$ a.s.\ contains no vertex of degree $i'$, 
and thus no fringe subtree $T_j$, so
$N_{T_{j}}(\mathcal{T}_{\mathbf{w},n}) =0$.
Consequently, \eqref{eq26} holds trivially in this case too.
Thus \eqref{eq26} holds in both cases.

Since
the sum in (\ref{eq26}) really only contains a finite number of terms,
it and \eqref{fi11} imply
that, 
as $n \rightarrow \infty$,
a.s., 
for some random vector $\bW$,
\begin{align} \label{eq27}
\left(\frac{ \E[N_{T_{1}}(\mathcal{T}_{\mathbf{w},n}) \mid
  \mathbf{n}_{\mathcal{T}_{\mathbf{w},n}}] - n
  \pi_{\theta(\mathbf{w})}(T_{1})}{\sqrt{n}}, \dots, \frac{
  \E[N_{T_{m}}(\mathcal{T}_{\mathbf{w},n}) \mid
  \mathbf{n}_{\mathcal{T}_{\mathbf{w},n}}]- n
  \pi_{\theta(\mathbf{w})}(T_{m})}{\sqrt{n}}\right) 
\to 
\bW\sim
{\rm N}(0, \Gamma^{\prime}_{\theta{(\mathbf{w})}}),  
\end{align}
\noindent where the covariance matrix
$\Gamma^{\prime}_{\theta{(\mathbf{w})}} =
(\gamma^{\prime}_{\theta{(\mathbf{w})}}(T_i, T_j))_{i,j=1}^{m}$ is given by,
for $T, T'\in\bbT$,
\begin{align} 
\gamma^{\prime}_{\theta{(\mathbf{w})}}(T, T') 
& = \pi_{\theta(\mathbf{w})}(T) \pi_{\theta(\mathbf{w})}(T') \sum_{r, r^{\prime} \geq 0} \frac{n_{T}(r) n_{T'}(r^{\prime})}{\theta_{r}(\mathbf{w}) \theta_{r^{\prime}}(\mathbf{w})} \dgamma(r, r^{\prime})  
\notag\\
& = \pi_{\theta(\mathbf{w})}(T) \pi_{\theta(\mathbf{w})}(T') 
\left(  - |T||T'| - \varsigma^{-2}_{\mathbf{w}} 
+ \sum_{r\geq 0}\frac{n_{T}(r)n_{T'}(r)}{\theta_{r}(\mathbf{w})}
\right)
.\end{align}
To obtain the second equality, we have
used \eqref{fi12}--\eqref{fi13} and (\ref{eq30}), i.e., $|T| = \sum_{r \geq 0} n_{T}(r) = 1+
\sum_{r \geq 0} r n_{T}(r)$, 
noting that it suffices to consider the case
$\pi_{\gth(\bw)}(T)\pi_{\gth(\bw)}(T')>0$, which implies that
$\gth_r(\bw)\gth_{r'}(\bw)>0$ when $n_T(r)n_{T'}(r')>0$.

Recall that (\ref{eq22}) holds conditioned on 
the sequence $(\mathbf{n}_{\mathcal{T}_{\mathbf{w},n}})_{n=1}^\infty$. 
Therefore, the limits (\ref{eq22}) and (\ref{eq27}) hold jointly, 
with independent limits. 
It follows that \eqref{eq28} holds with
$\xGamma_{\theta{(\mathbf{w})}} = \Gamma_{\theta{(\mathbf{w})}} + 
\Gamma^{\prime}_{\theta{(\mathbf{w})}}$, and a simple calculation
yields \eqref{fi4}--\eqref{fi5}.
\end{proof}

\refT{TCorollary1} gives a partial solution to \cite[Problem 21.4]{Svante2012},
but the general case remains open.
\begin{problem}
Does \eqref{eq28} in \refT{TCorollary1} hold
for any weight sequence $\mathbf{w}$, with some covariance matrix
$\xGamma_{\theta{(\mathbf{w})}} = (\xgamma_{\theta{(\mathbf{w})}}(T_{i},
T_{j}))_{i,j=1}^{m}$? 
If so, is
$\xGamma_{\theta{(\mathbf{w})}}$ given by
 \eqref{fi4}--\eqref{fi5}, for a suitable $\varsigma^2_\bw$?
\end{problem}

The argument used to answer the previous question in the cases \ref{case1GW}, \ref{case2GW} or \ref{case3GW} (second part of \refT{TCorollary1}) works as soon as one has a general version of Theorem \ref{TheoDegreeS}. 

\begin{problem}\label{ProbPaul}
Does \eqref{fi11} in Theorem \ref{TheoDegreeS} hold for any weight sequence
$\mathbf{w}$?
\end{problem}

This was conjectured in \cite{Paul2020} (at least for $\rho_\bw>0$), 
but remains open as far as we know.
See also \refR{RB1}.

\section{Application to additive functionals}  
\label{AdditiveFunct}

Let $f: \mathbb{T} \rightarrow \mathbb{R}$ be a functional of rooted trees  (in this context often called \emph{toll function}) and for $T \in \mathbb{T}$, consider the functional $F$ (often called an \emph{additive functional}) that is defined as the sum over all fringe subtrees 
\begin{align} \label{Feq1}
F(T) = F(T, f) \coloneqq \sum_{v \in T} f(T_{v}). 
\end{align}
\noindent In particular, by choosing $f(T) = \mathbf{1}_{\{T=T^{\prime}\}}$
for some  $T^{\prime} \in \mathbb{T}$, we obtain $F(T)=N_{T^{\prime}}(T)$. 
Moreover, for any $f$,
\begin{align} \label{Feq2}
F(T) = \sum_{T^{\prime} \in \mathbb{T}} f(T^{\prime}) N_{T^{\prime}}(T) 
,\end{align}
\noindent  i.e., $F(T)$ can be written as a linear combination of subtree
  counts $N_{T^{\prime}}(T)$.

For a probability distribution $\mathbf{p} = (p_{i})_{i \geq 0} \in \mathcal{P}_{1}(\mathbb{N}_{0})$ and a functional $f: \mathbb{T} \rightarrow \mathbb{R}$ such that 
\begin{align}
\sum_{T \in \mathbb{T}} |f(T)| \pi_{\mathbf{p}}(T) < \infty, \label{Feq3}
\end{align}
\noindent  we let 
\begin{align}
\E_{\pi_{\mathbf{p}}}[f(\mathcal{T})] \coloneqq 
\E [f(\cTp)]=
\sum_{T \in \mathbb{T}} f(T) \pi_{\mathbf{p}}(T) 
.\label{Feq4}
\end{align}

We say that the functional $f: \mathbb{T} \rightarrow \mathbb{R}$ has
\emph{finite support} if $f(T) \neq 0$ only for finitely many trees;
equivalently, there exists a constant $K >0$ such that $f(T) = 0$ unless
$|T| \leq K$. It then follows from (\ref{Feq2}) that the additive functionals
$F$ associated to $f$ with finite support are exactly the (finite) linear
combinations of subtree counts.  Theorem \ref{TheoCLT} implies the following
corollary. Note that a functional $f$ with finite support always satisfies
(\ref{Feq3}) for any distribution  $\mathbf{p} \in
\mathcal{P}_{1}(\mathbb{N}_{0})$; indeed, in this case, the left-hand side
of (\ref{Feq3}) is a sum with finitely many non-zero summands. 
\begin{theorem} \label{TCAddSupp}
Let $\mathbf{n}_{\kappa}$, $\kk\ge1$, 
be some degree statistics that satisfy Condition
\ref{Condition1} and let $\mathcal{T}_{\mathbf{n}_{\kappa}} \sim {\rm
  Unif}(\mathbb{T}_{\mathbf{n}_{\kappa}})$. Suppose that $f: \mathbb{T}
\rightarrow \mathbb{R}$ is a functional of rooted trees with finite
support, and let $F$ be the corresponding additive functional. 
Then, as $\kappa \rightarrow \infty$, 
\begin{align}
\E F(\mathcal{T}_{\mathbf{n}_{\kappa}}) & =  |\mathbf{n}_{\kappa}| \E_{\pi_{\mathbf{p}}}[f(\mathcal{T})] +  o(|\mathbf{n}_{\kappa}|),\label{Feq5}\\
\Var F(\mathcal{T}_{\mathbf{n}_{\kappa}}) 
& = |\mathbf{n}_{\kappa}| \gamma_{\mathbf{p}}(f) +  o(|\mathbf{n}_{\kappa}|), \label{Feq6} 
\end{align}
\noindent where
\begin{align} \label{Feq7} 
\gamma_{\mathbf{p}}(f) \coloneqq 
2\E_{\pi_{\mathbf{p}}}[f(\mathcal{T})F(\mathcal{T})] 
-\E_{\pip}[f(\cT)^2]
+ \left(\E_{\pi_{\mathbf{p}}}[f(\mathcal{T}) (|\mathcal{T}| -1)] \right)^{2} 
- \sum_{i \geq 0} \frac{1}{p_{i}} 
\left(\E_{\pi_{\mathbf{p}}}[f(\mathcal{T}) n_{\mathcal{T}}(i)] \right)^{2} 
\mathbf{1}_{\{p_{i} >0\}}
\end{align}
\noindent is finite,
with $0\le\gammap(f)<\infty$.
Furthermore,
\begin{align}
\frac{F(\mathcal{T}_{\mathbf{n}_{\kappa}}) -
  \E[F(\mathcal{T}_{\mathbf{n}_{\kappa}})]}{\sqrt{|\mathbf{n}_{\kappa}|}} &
\dto 
{\rm N}(0,\gamma_{\mathbf{p}}(f)),  \hspace*{4mm} \text{as} \hspace*{2mm} \kappa \rightarrow \infty. \label{Feq8} 
\end{align}
\end{theorem}
                                                                            
\begin{remark}\label{RF1}
We have excluded terms with $p_i=0$ in the sum in \eqref{Feq7};
these terms make no difference, since if $p_i=0$, then for every $T$ either 
$n_T(i)=0$ or $\pip(T)=0$, and 
thus $\E_{\pip}[f(\cT)n_T(i)]=0$.
Furthermore, this sum 
has only finitely many non-zero terms, since we only need to consider
$i\in\bigcup_{T:f(T)\neq0} \cD(T)$, which is a finite union of finite sets.
\end{remark}

\begin{proof}
Since $f$ has finite support, there exists a constant $K >0$ such that $f(T)
= 0$ unless $|T| \leq K$, for $T \in \mathbb{T}$. 
Let $\bbTK:=\set{T\in\bbT:|T|\le K}$, and note that $\bbTK$ is a finite set
of trees.
It
follows from (\ref{Feq1}) and (\ref{Feq2}) that the corresponding additive
functional $F$ is given by 
\begin{align} \label{Feq9}
F(T) =  \sum_{T^{\prime} \in \bbTK} f(T^{\prime}) N_{T^{\prime}}(T), \hspace*{4mm} T \in \mathbb{T};
\end{align}
\noindent note that (\ref{Feq9}) has finitely many summands. 
We label the elements of $\bbTK$ as $T_1,\dots,T_m$, for some $m \in \mathbb{N}$, and
apply \refT{TheoCLT}.

Since $f$ has
finite support, it satisfies (\ref{Feq3}), 
and thus
$\E_{\pi_{\mathbf{p}}}[f(\mathcal{T})]$ is defined (and finite). 
Then, as $\kappa\rightarrow \infty$, (\ref{Feq5}) follows from 
\eqref{Feq9},
(\ref{eq6}) in Theorem \ref{TheoCLT}, and (\ref{Feq4}), which yield
\begin{align} \label{Feq10}
\E F(\mathcal{T}_{\mathbf{n}_{\kappa}}) 
=  \sum_{T \in \bbTK} f(T) \E N_{T}(\mathcal{T}_{\mathbf{n}_{\kappa}}) 
= |\mathbf{n}_{\kappa}|  \sum_{T \in \bbTK} f(T) \pi_{\mathbf{p}}(T) +
  o(|\mathbf{n}_{\kappa}|)
= |\bnk|\E_{\pi_{\mathbf{p}}}[f(\mathcal{T})]+o(|\bnk|)
.\end{align}


Similarly, (\ref{Feq9}) and 
(\ref{eq15})--(\ref{eq13}) in Theorem \ref{TheoCLT} imply
that, as $\kappa \rightarrow \infty$,  
\begin{align} 
\Var F(\mathcal{T}_{\mathbf{n}_{\kappa}})
& =  \sum_{T \in \bbTK} \sum_{T^{\prime}\in \bbTK} f(T) f(T^{\prime}) 
\Cov \bigpar{N_{T}(\mathcal{T}_{\mathbf{n}_{\kappa}}), N_{T^{\prime}}(\mathcal{T}_{\mathbf{n}_{\kappa}}) } \nonumber \\
& = |\mathbf{n}_{\kappa}| 
\sum_{T \in \bbTK} \sum_{T^{\prime} \in \bbTK} f(T) f(T^{\prime})
\gamma_{\mathbf{p}}(T, T^{\prime}) 
+ o(|\mathbf{n}_{\kappa}|) \label{Feq12}, 
\end{align}
\noindent where $\gamma_{\mathbf{p}}(T, T^{\prime})$ is defined in
(\ref{eq20})--(\ref{eq21}). 
In other words, \eqref{Feq6} holds with
\begin{align}\label{Fe100}
  \gammap(f):=
\sum_{T \in \bbTK} \sum_{T^{\prime} \in \bbTK} f(T) f(T^{\prime})
\gamma_{\mathbf{p}}(T, T^{\prime}) 
.\end{align}
Moreover, (\ref{Feq9}) and (\ref{Feq10}) imply that
\begin{align} \label{Feq11}
F(\mathcal{T}_{\mathbf{n}_{\kappa}}) -
  \E F(\mathcal{T}_{\mathbf{n}_{\kappa}}) 
=  \sum_{T \in \bbTK} f(T) \left(N_{T}(\mathcal{T}_{\mathbf{n}_{\kappa}}) -
  \E N_{T}(\mathcal{T}_{\mathbf{n}_{\kappa}}) \right),
\end{align}
which together with \eqref{eq7} implies that \eqref{Feq8} holds with
the same $\gammap(f)$ given by \eqref{Fe100}.

It remains to evaluate this $\gammap(f)$ and show that \eqref{Fe100} agrees
with \eqref{Feq7}.
First, \eqref{Feq4} yields
\begin{align}\label{Fe101}
 \sum_{T\in\bbTK} f(T)^2 \pi_{\mathbf{p}}(T) 
= \E_{\pip}\sqpar{f(\cT)^2}.
\end{align}
Next,
observe that \eqref{Feq4}  and \eqref{Feq9} yield
\begin{align} 
 \sum_{T \in \bbTK} \sum_{T^{\prime} \in \bbTK} f(T) f(T^{\prime})
  N_{T^{\prime}}(T) \pi_{\mathbf{p}}(T) & 
= \sum_{T^{\prime} \in \bbTK}  f(T^{\prime})  
\sum_{T \in \bbTK}  f(T) N_{T^{\prime}}(T) \pi_{\mathbf{p}}(T) \nonumber \\
& = \sum_{T^{\prime} \in \bbTK}  f(T^{\prime}) 
\E_{\pi_{\mathbf{p}}}[f(\mathcal{T}) N_{T^{\prime}}(\mathcal{T})] \nonumber  \\
& = \E_{\pi_{\mathbf{p}}}[f(\mathcal{T})F(\mathcal{T})]. \label{Feq13}
\end{align}
Hence, recalling \eqref{Fe101}, since $N_T(T)=1$ for every tree $T$,
\begin{align}\label{Fe102}
   \sum_{T\neq T'} f(T) f(T^{\prime}) N_{T^{\prime}}(T) \pi_{\mathbf{p}}(T) 
&= \E_{\pi_{\mathbf{p}}}[f(\mathcal{T})F(\mathcal{T})]
- \sum_{T\in\bbTK} f(T)^2 \pi_{\mathbf{p}}(T) 
\notag\\&
= \E_{\pi_{\mathbf{p}}}[f(\mathcal{T})F(\mathcal{T})]
- \E_{\pip}\sqpar{f(\cT)^2}
.\end{align}
Furthermore, by the definition \eqref{eq12} of $\eta_{\mathbf{p}}(T,T^{\prime})$ 
and (\ref{Feq4}), 
\begin{align} 
& \sum_{T \in \bbTK} \sum_{T^{\prime} \in \bbTK} f(T) f(T^{\prime}) \eta_{\mathbf{p}}(T, T^{\prime}) \pi_{\mathbf{p}}(T) \pi_{\mathbf{p}}(T^{\prime}) \nonumber  \\
&\hspace*{10mm} = \left(\sum_{T \in \bbTK}  f(T) (|T|-1)
\pi_{\mathbf{p}}(T) \right)^{2} 
- \sum_{i \geq 0} \frac{1}{p_{i}} \left( \sum_{T \in \bbTK}  f(T) n_{T}(i)
\pi_{\mathbf{p}}(T) \right)^{2} 
\nonumber  \\
&\hspace*{10mm} 
= \left(\E_{\pi_{\mathbf{p}}}[f(\mathcal{T}) (|\mathcal{T}|-1)] \right)^{2} 
- \sum_{i \geq 0} \frac{1}{p_{i}} \left(\E_{\pi_{\mathbf{p}}}[f(\mathcal{T}) n_{\mathcal{T}}(i)] \right)^{2}\mathbf{1}_{\{p_{i} >0\}} \label{Feq14}
,\end{align}
where we recall \refR{RF1}.
Note also that all sums and expectations in \eqref{Fe101}--\eqref{Feq14} are
finite, since $f$ has finite support.

Finally,  by combining \eqref{Fe100} and \eqref{eq20}--\eqref{eq21},
 with  \eqref{Fe101}, \eqref{Fe102}, and  (\ref{Feq14}), we
obtain (\ref{Feq7}). We have already remarked that all terms in \eqref{Feq7}
are finite.
\end{proof}

Note that we have not ruled out the possibility that $\gammap(f)=0$.
This may happen in trivial cases; whether it may happen in non-trivial cases
(properly defined)
is equivalent to \refP{ProbGamma}.
We may restate that problem in a slightly stronger form:
\begin{problem}
Let $\bp\in\ppi$ be given.
For which functionals $f$ with finite support is $\gammap(f)=0$?       
\end{problem}

In analogy to \cite{Svante2016} (which treats conditioned Galton--Watson trees),
we may, more generally, also study 
additive functionals where 
the associated toll function does not necessarily have a finite support. 
We use a standard truncation argument.
Let $f: \mathbb{T} \rightarrow \mathbb{R}$ be an arbitrary toll function.
For a constant $K>0$, define the truncation and tail functionals associated
to $f$ by letting 
\begin{align} \label{Feq15}
f^{(K)}(T) \coloneqq f(T) \mathbf{1}_{\{|T| \leq K \}} \hspace*{3mm} \text{and} \hspace*{3mm} \hf^{(K)}(T) \coloneqq f(T) - f^{(K)}(T), \hspace*{3mm} T \in \mathbb{T}. 
\end{align}
\noindent For $T \in \mathbb{T}$, we also let $F^{(K)}(T) \coloneqq F(T, f^{(K)})$ and $\hF^{(K)}(T) \coloneqq F(T, \hf^{(K)})$ be the additive functionals associated to $f^{(K)}$ and $\hf^{(K)}$, respectively; see (\ref{Feq1}).

\begin{theorem} \label{TCAddSupp2}
Let $\mathbf{n}_{\kappa}$, $\kk\ge1$, 
be some degree statistics that satisfy Condition \ref{Condition1} and let 
$\mathcal{T}_{\mathbf{n}_{\kappa}} \sim {\rm Unif}(\mathbb{T}_{\mathbf{n}_{\kappa}})$. Suppose that $f: \mathbb{T} \rightarrow \mathbb{R}$ is a functional of rooted trees such that the following conditions are satisfied:
\begin{enumerate}[label=\upshape(\roman*)]
\item There exists a real constant $0 \leq \gamma < \infty$ such that $\displaystyle \lim_{K \rightarrow \infty} \gamma_{\mathbf{p}}(f^{(K)}) = \gamma$; \label{AdCond1}
\item $\lim_{K \rightarrow \infty} \limsup_{\kappa \rightarrow \infty} \frac{{\rm Var}(\hF^{(K)}(\mathcal{T}_{\mathbf{n}_{\kappa}}))}{|\mathbf{n}_{\kappa}|}=0$.  \label{AdCond2}
\end{enumerate}
\noindent Then,
for the  corresponding additive functional $F$,
\begin{align} \label{Feq16}
\frac{F(\mathcal{T}_{\mathbf{n}_{\kappa}}) - \E[F(\mathcal{T}_{\mathbf{n}_{\kappa}})]}{\sqrt{|\mathbf{n}_{\kappa}|}} & \rightarrow {\rm N}(0,\gamma),  \hspace*{4mm} \text{as} \hspace*{2mm} \kappa \rightarrow \infty. 
\end{align}
\end{theorem}
\begin{proof}
Note that (\ref{Feq8}) in  \refT{TCAddSupp} applies to each $f^{(K)}$, as $\kappa \rightarrow \infty$. Then, \ref{AdCond1}, \ref{AdCond2} and \cite[Theorem 3.2]{Billingsley1999} (or \cite[Theorem 4.28]{Kallenberg2002}) imply that we can let $K \rightarrow \infty$ and conclude (\ref{Feq16}). 
\end{proof}

The truncation argument used in the proof of 
\refT{TCAddSupp2} is the same as in \cite{Svante2016}, where critical
conditioned Galton--Watson trees with finite offspring variance were studied
(under some conditions on the toll functions). Indeed, the conditions in
\cite[Theorem 1.5]{Svante2016} say roughly that the functional $f(T)$ is
small when $|T|$ is large, and the proof there consists of verifying
(under these conditions) results analogous
to \ref{AdCond1} and \ref{AdCond2} in \refT{TCAddSupp2}. 
Furthermore, in \cite[Theorem 1.13]{Svante2016},
asymptotic normality was also proved (in the same way) 
for additive functionals under the
assumption that the toll function $f(T)$ is bounded and local (i.e., only
depends on a fixed neighbourhood of the root of $T$). 
This was extended further by
Ralaivaosaona et al.\ \cite{Ralaivaosaona2020}, who extended
the result 
to  ``almost local''
functionals (but assuming higher moments of the offspring distribution).  
For applications of our results above for random trees with given degree
statistics,  
it would be useful to have some explicit sufficient conditions
on $f$, 
similar to the conditions in the references just mentioned for conditioned
Galton--Watson trees.

 \begin{problem}
Find suitable conditions on 
the functional $f$,
and perhaps also on
the degree statistics $\mathbf{n}_{\kappa}$,
such that \ref{AdCond1} and \ref{AdCond2} in 
\refT{TCAddSupp2} are satisfied, and thus \eqref{Feq16} holds.
 \end{problem}

\begin{appendices}

\section{A multidimensional Gao--Wormald theorem} \label{AppendixA}

In this section, we generalize the Gao--Wormald theorem 
\cite[Theorem 1]{Gao2004} to the multidimensional setting. 
Note that a different but closely related multidimensional generalization
of the Gao--Wormald theorem has been shown recently (and independently)
by Hitczenko and Wormald \cite{HitczenkoWormald};
the two multidimensional versions use essentially the same condition on
(high) factorial moments, but the conditions and the result are stated in 
different ways.
It is not clear exactly how the two versions are related,
but it seems that the version in \cite{HitczenkoWormald} is more flexible,
and gives more precise results in cases when our asymptotic covariance
matrix $\Gamma$ is singular. On the other hand, it seems that our version
is easier to apply in the present paper, and perhaps also in some
other applications.
(We thank the authors of \cite{HitczenkoWormald}
for interesting discussions on multidimensional
versions of the Gao--Wormald theorem.)

For a sequence of real-valued random variables $(Z_{n})_{n \geq 1}$, and a sequence of real numbers $(a_{n})_{n \geq 1}$ such that $a_{n} > 0$, we write $Z_{n} = o_{\rm p}(a_{n})$ when $Z_{n}/a_{n} \rightarrow 0$, as $n \rightarrow \infty$, in probability. For a complex number $x$ we denote by ${\rm Re}(x)$ its real part. 

\begin{theorem}\label{TGW}
For $m, n \in \mathbb{N}$, let $(X_{1n},\dots,X_{mn})$ be vectors of non-negative random variables. Suppose that $\mu_{in}$ and $\sigma_{in}$ are positive real numbers such that for each $1 \leq i \leq m$, as $n \rightarrow \infty$,
\begin{align}\label{tgw1}
\sigma_{in} \ll \mu_{in} \ll \sigma_{in}^{3}.
\end{align}
\noindent Let $\Gamma \coloneqq (\gamma_{ij})_{i,j=1}^m$ be a fixed matrix.
Let $c>0$ be a constant, and suppose further that, as $n \rightarrow \infty$, 
uniformly for all integer sequences $(k_{in})_{i=1}^{m}$ with $0 \leq k_{in}
\leq c \mu_{in}/\sigma_{in}$,
\begin{align}\label{tgw2}
\E \prod_{i=1}^{m} (X_{in})_{k_{in}} = \prod_{i=1}^{m} \mu_{in}^{k_{in}} \cdot \exp \left(\frac{1}{2} \sum_{i,j=1}^{m} \frac{\gamma_{ij} \sigma_{in} \sigma_{jn}-\delta_{ij} \mu_{in}}{\mu_{in} \mu_{jn} }k_{in}k_{jn} + o(1)\right),
\end{align}
\noindent Then, 
\begin{align}\label{tgw3}
\left( \frac{X_{1n}- \mu_{1n}}{\sigma_{1n}},  \dots, \frac{X_{mn}- \mu_{mn}}{\sigma_{mn}} \right) \dto {\rm N}(0, \Gamma), \hspace*{4mm} \text{as} \hspace*{2mm} n \rightarrow \infty. 
\end{align}
\end{theorem}

\begin{proof}
We follow closely the proof of the one-dimensional version in \cite{Gao2004}. In the following, unspecified limits are as $n \rightarrow \infty$. Note first that \eqref{tgw1} implies $\sigma_{in} \rightarrow \infty$ and $\mu_{in} \rightarrow \infty$. For $1 \leq i \leq m$, we let
\begin{align}
\zeta_{in} & \coloneqq \frac{\mu_{in}}{\sigma_{in}} \ln \left( \frac{X_{in}}{\mu_{in}}\right), \label{tq1} \\
Q_{in}(k) & \coloneqq \frac{(X_{in})_{k}}{(\mu_{in})_{k}}, \label{tq2} \\
t_{in } & \coloneqq k_{in}\frac{\sigma_{in}}{\mu_{in}}.  \label{tq3}
\end{align}
\noindent By Lemma \ref{L0} and \eqref{tgw1}, the assumption \eqref{tgw2} implies, uniformly for allowed sequences $(k_{in})_{i=1}^{m}$,
\begin{align}\label{tq4}
\E \prod_{i=1}^{m} Q_{in}(k_{in}) = \exp \left(\frac{1}{2} \sum_{i,j=1}^{m} \frac{\gamma_{ij} \sigma_{in} \sigma_{jn} k_{in}k_{jn}}{\mu_{in} \mu_{jn} }  + o(1) \right) = \exp \left(\frac{1}{2} \sum_{i,j=1}^{m} \gamma_{ij} t_{in} t_{jn} \right) + o(1). 
\end{align}
\noindent Note that \eqref{tq3} implies that $0\leq t_{in} \leq c$, and thus $t_{in} =O(1)$. In particular, the expectation in \eqref{tq4} is $O(1)$.

We recall the inequality \cite[(2.8)]{Gao2004}, which says that if $a \geq b >k$, then
\begin{align}\label{tq5}
k \ln(a/b) \leq \ln \left( \frac{(a)_k}{(b)_k} \right) \leq k \log \left(\frac{a-k}{b-k}\right) \leq \frac{k\ln(a/b)}{1-k/b}.
\end{align}
\noindent Note that $k_{in}/\mu_{in} \leq c/ \sigma_{in} \rightarrow 0$, for
$1 \leq i \leq m$.  In the sequel, we will consider only $n$ that are
large enough so that $\sigma_{in} >2c$, for every $1 \leq i \leq m$. Hence,
for allowed sequences $(k_{in})_{i=1}^{m}$, we have $k_{in}/\mu_{in}<1/2$.
Then, 
when $X_{in} \geq \mu_{in}$,  
\eqref{tq5} yields
$2\ln Q_{in}(k_{in}) \leq 4 k_{in} \ln(X_{in}/\mu_{in}) \leq \ln Q_{in}(4 k_{in})$. Thus, when $X_{in} \geq \mu_{in}$,
\begin{align}\label{tq7}
 Q_{in}(k_{in})^2 \leq  Q_{in}(4k_{in}).
\end{align}
Furthermore, if $k_{in} \leq X_{in} \leq \mu_{in}$, we trivially have
$Q_{in}(k_{in})\leq 1$. 
Similarly, it is easy to see that if $0\leq
X_{in}<k_{in}$, then $|(X_{in})_{k_{in}}|\leq (k_{in})_{k_{in}}\leq
(\mu_{in})_{k_{in}}$ and thus $|Q_{in}(k_{in})|\leq 1$ in this case
too. (This is trivial if $X_{in}$ is integer-valued.) By combining these
observations with \eqref{tq7} for the case $X_{in}\ge\mu_{in}$, 
we see that for any allowed sequences
$(k_{in})_{i=1}^{m}$, and any $X_{in}$,
\begin{align}\label{tq8}
Q_{in}(k_{in})^{2} \leq  Q_{in}(4k_{in}) + 1 = Q_{in}(4k_{in}) + Q_{in}(0).
\end{align}

We set $c^{\prime} \coloneqq c/4$, and we henceforth consider only sequences $(k_{in})_{i=1}^{m}$ such that
\begin{align}\label{tq00}
k_{in} \leq c^{\prime}\mu_{in}/\sigma_{in}.
\end{align}
Then \eqref{tgw2} and, as a consequence, \eqref{tq4} hold
also with $4k_{in}$ or $0$ instead of $k_{in}$.
Hence, \eqref{tq8} implies that
\begin{align}\label{tq9}
\E  \prod_{i=1}^{m} \left( 1+Q_{in}(k_{in})^{2} \right) 
\leq C,
\end{align}
for some constant $C$.
(We similarly use $C$  below to denote unknown constants, possibly with
different values on each occasion.) 

For $1 \leq i \leq m$, we set $\varepsilon_{in} \coloneqq 2c^{\prime}/\sigma_{in}$. Then, \eqref{tq00} implies $k_{in}/(\mu_{in}/2)\leq \varepsilon_{in}$.
Note that $\varepsilon_{in}<2c^{\prime}/(2c)= 1/4$ and $\varepsilon_{in} \rightarrow 0$, as $n \rightarrow \infty$. It is shown in \cite{Gao2004} that
\begin{align}\label{tq10}
\left| e^{t_{in} \zeta_{in}}-Q_{in}(k_{in}) \right| 
\le 
\left\{ \begin{array}{lcl}
             \min( Q_{in}(k_{in}), \varepsilon_{in} Q_{in}(k_{in}) \ln (Q_{in}(k_{in})))  & \mbox{  if } &  X_{in} \geq \mu_{in}, \\
            \varepsilon_{in} Q_{in}(k_{in})^{1-\varepsilon_{in}} \ln(1/Q_{in}(k_{in}))   & \mbox{  if } & \mu_{in}/2 \leq X_{in} < \mu_{in}, \\
            2^{-k_{in}} & \mbox{  if } & 0 \leq X_{in} < \mu_{in}/2. \\
              \end{array}
                  \right.
\end{align}
Note also that it follows from \eqref{tq2} that $X_{in}\ge \mu_{in}$ if and
only if $Q_{in}\ge1$.
By considering the four cases $X_{in} < \mu_{in}/2$,
$\mu_{in}/2\leq X_{in}<\mu_{in}$, $1\leq Q_{in} \leq \varepsilon_{in}^{-1/2}$ and
$Q_{in}>\varepsilon_{in}^{-1/2}$, it follows from \eqref{tq11} that
\begin{align}\label{tq11}
\left| e^{t_{in} \zeta_{in}}-Q_{in}(k_{in}) \right| & \leq 2^{-k_{in}} + C \varepsilon_{in}+ C \varepsilon_{in}^{1/2} \ln(\varepsilon_{in}^{-1/2}) + Q_{in}(k_{in}) \mathbf{1}_{\{Q_{in}(k_{in}) > \varepsilon_{in}^{-1/2}\}} \notag \\
& \leq 2^{-k_{in}}+C \varepsilon_{in}^{1/3} + \varepsilon_{in}^{1/2} Q_{in}(k_{in})^{2} \notag \\ 
& \leq C (2^{-k_{in}} + C \varepsilon_{in}^{1/3})(1+Q_{in}(k_{in})^{2})
.\end{align}
Furthermore, \eqref{tq5} shows that if $X_{in}\geq \mu_{in}$, then $e^{t_{in}\zeta_{in}}=(X_{in}/\mu_{in})^{k_{in}}\leq Q_{in}(k_{in})$. Hence, for any $X_{in}$,
\begin{align}\label{tq12}
  0<e^{t_{in}\zeta_{in}} \leq 1+ Q_{in}(k_{in}).
\end{align}
\noindent It follows from \eqref{tq11} and \eqref{tq12} that
\begin{align}\label{tq13}
\left| \prod_{i=1}^{m} e^{t_{in}\zeta_{in}}- \prod_{i=1}^{m} Q_{in}(k_{in}) \right| & \leq \sum_{i=1}^{m} \left| e^{t_{in} \zeta_{in}}-Q_{in}(k_{in}) \right| \prod_{j \neq i}(1+Q_{jn}(k_{jn})) \notag \\
& \leq C \sum_{i=1}^{m} (2^{-k_{in}} + C \varepsilon_{in}^{1/3}) \cdot \prod_{i=1}^{m} (1+Q_{in}(k_{in})^{2})
.\end{align}
Now, we further restrict to sequences $(k_{in})_{i=1}^{m}$ such that $c^{\prime \prime} \mu_{in}/\sigma_{in} \leq k_{in} \leq
c^{\prime}\mu_{in}/\sigma_{in}$, where $c^{\prime \prime} \coloneqq c^{\prime}/2$.
Then $2^{-k_{in}}\rightarrow 0$, uniformly for all allowed sequences $(k_{in})_{i=1}^{m}$. Thus, \eqref{tq13} implies, uniformly,
\begin{align}\label{tq14}
\left| \prod_{i=1}^{m} e^{t_{in}\zeta_{in}}- \prod_{i=1}^{m} Q_{in}(k_{in}) \right|   
= o(1) \cdot \prod_{i=1}^{m} (1+Q_{in}(k_{in})^{2}). 
\end{align}
\noindent It then follows from \eqref{tq9} and \eqref{tq4} that, uniformly for all allowed sequences $(k_{in})_{i=1}^{m}$,
\begin{align}\label{tq15}
\E \prod_{i=1}^{m} e^{t_{in}\zeta_{in}}  = \exp \left(\frac{1}{2} \sum_{i,j=1}^{m} \gamma_{ij} t_{in} t_{jn} \right) + o(1).
\end{align}

We claim that \eqref{tq15} implies that 
\begin{align}\label{ky6}
(\zeta_{1n}, \dots, \zeta_{mn}) \dto {\rm N}(0, \Gamma), \hspace*{4mm} \text{as} \hspace*{2mm} n \rightarrow \infty. 
\end{align}
\noindent Then, by the definition \eqref{tq1}, since $\zeta_{in}$ is tight by \eqref{ky6} and $\sigma_{in}/\mu_{in} \rightarrow 0$, 
\begin{align}\label{ky7}
\frac{X_{in}-\mu_{in}}{\sigma_{in}} 
= \frac{\mu_{in}}{\sigma_{in}} \left(
  \exp\left(\frac{\sigma_{in}}{\mu_{in}}\zeta_{in}\right)-1 \right) 
= \zeta_{in}+o_{\rm p}(1),
\end{align}
\noindent for $1 \leq i \leq m$. Thus, the result \eqref{tgw3} follows from \eqref{ky6}.

The rest of the proof is dedicated to prove the claim in \eqref{ky6}. The
proof is routine, although we do not know any reference. 
So, we  include the argument for completeness.

Note that \eqref{tq15} holds for all $t_{in} \in [c^{\prime \prime},c^{\prime}]$ such that $t_{in} \mu_{in}/\sigma_{in}$ is an integer. By assumption, $\mu_{in}/\sigma_{in} \rightarrow \infty$, and thus the gaps between the allowed $t_{in}$ tend to $0$ as $n \rightarrow \infty$; in particular, there exist such $t_{in}$ for large enough $n$.
Define $Y_{in} \coloneqq e^{\zeta_{in}}$, for $1 \leq i \leq m$. It follows
easily from \eqref{tq15} (or, rather, the one-dimensional version, which
follows as above) that the sequences $(Y_{in})_{n \geq 1}$ are tight,
for $1 \leq i \leq m$, and thus so is the sequence of random vectors
$(Y_{1n}, \dots, Y_{mn})_{n \geq 1}$. Consider a subsequence such that
$(Y_{1n}, \dots, Y_{mn}) \dto (Y_{1}, \dots, Y_{m})$, for some random
variables $Y_1,\dots,Y_m$. Set $c^{\prime \prime \prime} \coloneqq 0.8
c^{\prime}$, say. 
Fix any real numbers 
$t_{i} \in [c^{\prime \prime},c^{\prime \prime  \prime}]$, for $1 \leq i \leq m$.
For sufficiently large $n$, 
we may find
$t_{in} \in [c^{\prime \prime},c^{\prime \prime \prime}]$ such that $t_{in}
\rightarrow t_i$ and $t_{in} \mu_{in}/\sigma_{in}$ are integers,
and also $t'_{in} \in [c'',1.1c'']$
and $t''_{in} \in [0.9 c^{\prime},c^{\prime}]$ such that
$t'_{in} \mu_{in}/\sigma_{in}$ and
$t''_{in} \mu_{in}/\sigma_{in}$ are integers.
Then, $t'_{in}\le 1.1 t_{in} \le t''_{in}$, and thus, using \eqref{tq15},
\begin{align}
  \E \lrpar{\prod_{i=1}^m Y_{in}^{t_{in}}}^{1.1}
=\E \lrsqpar{\prod_{i=1}^m Y_{in}^{1.1t_{in}}}
\le   \E \lrsqpar{\prod_{i=1}^m \bigpar{Y_{in}^{t'_{in}}+Y_{in}^{t''_{in}}}}
\le C.
\end{align}
Hence, the sequence $\prod_{i=1}^m Y_{in}^{t_{in}}$ is uniformly integrable.
 Furthermore $\prod_{i=1}^{m} Y_{in}^{t_{in}} \dto \prod_{i=1}^{m} Y_{i}^{t_i}$
 (along the subsequence), 
and thus it follows from \eqref{tq15} that
\begin{align}\label{ky1}
\E  \prod_{i=1}^{m} Y_{i}^{t_{i}}   = \lim_{n \rightarrow \infty}\E
  \prod_{i=1}^{m}  Y_{in}^{t_{in}}  = \exp \left(\frac{1}{2}
  \sum_{i,j=1}^{m} \gamma_{ij} t_{i} t_{j} \right)
.\end{align}
This holds for any $t_1,\dots,t_m \in [c^{\prime \prime},c^{\prime \prime \prime}]$.
 Since the expectation on the left-hand side of \eqref{ky1} thus is finite
 for these 
$t_1,\dots,t_m$, it follows that it is finite for all complex
$t_1,\dots,t_m$ with real parts in $(c^{\prime \prime},c^{\prime \prime \prime})$, and that it is a bounded analytic function in this domain. By analytic continuation, we thus have
\begin{align}\label{ky2}
\E \prod_{i=1}^{m} Y_{i}^{t_{i}}  = \exp \left(\frac{1}{2} \sum_{i,j=1}^{m} \gamma_{ij} t_{i} t_{j} \right), \hspace*{4mm} {\rm Re}(t_1),\dots, {\rm Re}(t_m) \in (c^{\prime \prime},c^{\prime \prime \prime}).
\end{align}
\noindent Furthermore, by taking $t_i=\frac12(c^{\prime \prime}+c''')+ \mathrm{i} u_i$ with real numbers $u_i$, it follows from the boundedness of \eqref{ky2} that the matrix $\Gamma \coloneqq (\gamma_{ij})_{i,j=1}^m$ is positive semi-definite. Hence the multivariate normal distribution ${\rm N}(0, \Gamma)$ exists. For $1 \leq i \leq m$, let $\zeta_{i} \coloneqq \ln (Y_i) \in [-\infty,\infty)$, and let $(\hat{\zeta}_1,\dots,\hat{\zeta}_m)\sim {\rm N}(0, \Gamma)$. Then, \eqref{ky2} can be written as
\begin{align}\label{ky3}
\E  \prod_{i=1}^{m} e^{t_{i} \zeta_{i}}  = \E  \prod_{i=1}^{m} Y_{i}^{t_{i}}   = \E \prod_{i=1}^{m} e^{t_{i} \hat{\zeta}_{i}} , \hspace*{4mm} {\rm Re}(t_1),\dots, {\rm Re}(t_m) \in (c^{\prime \prime},c^{\prime \prime \prime}).
\end{align}
\noindent Let $\nu$ be the distribution of $(\zeta_1,\dots,\zeta_m)$; note that this is a probability measure on $[-\infty,\infty)^m$. Also, let $\hat{\nu}$ be the distribution ${\rm N}(0, \Gamma)$. Fix some $\tau \in (c^{\prime \prime},c^{\prime \prime \prime})$, and define the conjugated measures $\nu^{\ast}$ and $\hat{\nu}^{\ast}$ on $\mathbb{R}^{m}$ by 
\begin{align}\label{ky4}
\frac{{\rm d} \nu^{\ast}}{{\rm d} \nu}(x_1,\dots,x_m) = \frac{{\rm d} \hat{\nu}^{\ast}}{{\rm d} \hat{\nu}}(x_1,\dots,x_m)= e^{\sum_{i=1}^{m} \tau x_i},  \hspace*{4mm} x_{1}, \dots, x_{m} \in \mathbb{R}. 
\end{align}
\noindent Then \eqref{ky3} (with ${\rm Re}(t_{i})=\tau$ for each $1 \leq i \leq m$) implies that the finite measures $\nu^{\ast}$ and $\hat{\nu}^{\ast}$ have the same Fourier transform; consequently, $\nu^{\ast} = \hat{\nu}^{\ast}$. This implies that the measures $\nu$ and $\hat{\nu}$ coincide on $\mathbb{R}^{m}$. Since $\hat{\nu}$ is a probability measure, this shows that $\nu(\mathbb{R}^{m})=1$, and
thus $\nu$ is concentrated on $\mathbb{R}^{m}$. In other words, $\zeta_{i}>-\infty$ a.s., and $(\zeta_1,\dots,\zeta_m)\sim {\rm N}(0, \Gamma)$. 

Consequently, we have shown that 
\begin{align}
 (e^{\zeta_{1n}}, \dots, e^{\zeta_{mn}})  
=(Y_{1n},\dots,Y_{mn})
\dto (Y_{1},\dots,Y_{m})\eqd
(e^{\hat{\zeta}_{1}}, \dots, e^{\hat{\zeta}_{m}}) , \hspace*{4mm} \text{as} \hspace*{2mm} n \rightarrow \infty,
\end{align}
\noindent along every convergent subsequence, and thus for the full sequence. Then, the claim \eqref{ky6} follows by taking logarithms. 
\end{proof}

\begin{remark}
  The original theorem by Gao and Wormald \cite[Theorem 1]{Gao2004} 
is essentially the case $m=1$ of \refT{TGW}.
(The condition $\gs_{in}\ll\mu_{in}$ is omitted in \cite{Gao2004} by
mistake, but it is clearly needed.)
However, there is a technical difference in that they assume \eqref{tgw2}
(for $m=1$) only for $c_1 \mu_n/\gs_n\le k_n \le c \mu_n/\gs_n$, for some
arbitrary $0<c_1<c$, while we assume it for all $k_n\le c \mu_n/\gs_n$.

In fact, the proof above, with minor modifications, shows that 
in the multivariate version in \refT{TGW}, it suffices to 
assume \eqref{tgw2}
for $k_{in}\in [c_1 \mu_{in}/\gs_{in}, c \mu_{in}/\gs_{in}]\cup\set0$,
i.e., it suffices to assume, for each $i$, that either $k_i=0$ or $k_i$ is
in a range as in \cite{Gao2004}.
In other words, it suffices to restrict $k_{in}$ as in \cite{Gao2004} if we
also assume that \eqref{tgw2} holds for the vector $(X_{in})_{i\in J}$, 
for every subset $J$ of $\set{1,\dots,m}$.
\end{remark}

\section{Proof of Theorem \ref{TheoDegreeS}} 
\label{AppendixB}

In this appendix, we give a new, simple, proof of
\refT{TheoDegreeS} under condition \ref{case1GW} or \ref{case2GW} 
there (stated in \refT{TCorollary1}), by
using the multivariate Gao--Wormald
theorem in \refApp{AppendixA}.

\begin{proof}[Proof of \refT{TheoDegreeS} assuming \ref{case1GW}  or
  \ref{case2GW}] 

We consider first a general weight sequence $\bw$, assuming only
$\rho_\bw>0$. 
By \refR{RGWequiv}, we may replace $\bw$ by the probability distribution
$\theta(\bw)$. In other words, we may and will assume that $\bw$ is a
probability distribution on $\bbNo$; thus $\cT_{\bw,n}$ is a 
conditioned Galton--Watson tree
with this offspring distribution.

Let $\xi_1,\xi_2,\dots$ be a sequence of independent random variables with
distribution $\bw$, and define the partial sums
\begin{align}\label{bb1}
  S_n:=\sum_{j=1}^n \xi_j,
\qquad n\ge0,
\end{align}
and
\begin{align}\label{bb1a}
  \tS_n:=\sum_{j=1}^n (\xi_j-1)=S_n-n,
\qquad n\ge0.
\end{align}

There is a well-known bijection between $\bbT_n$ and the set $\bbE^n$ of
excursions of length $n$, which combines the bijections 
between $\bbT_\bn$ and $\bbE_\bn$ mentioned in \refS{Strees}
for different $\bn$ with $|\bn|=n$.
Let us denote this bijection by $\gU$.
By \eqref{tw5}, for any tree $T\in\bbT_n$, the corresponding excursion 
$\gU(T)\in\bbE^n$
has increments that are the vertex degrees minus $1$, i.e., 
$d_T(i)-1$, $i=1\dots,n$, in some order.
It follows that for the conditioned Galton--Watson tree $\cT_{\bw,n}$,
the random excursion $\gU(\cT_{\bw,n})\in\bbE^n$ 
has a distribution that equals the
distribution of $\xS_n:=(\tS_0,\dots,\tS_n)$
conditioned on $\xS_n\in\bbE^n$.
This means that the degree statistic $\bn_{\cT_{\bw,n}}$ has the same distribution
as the sequence $\bigpar{|\set{1\le j\le n: \xi_j=i}|:i\ge0}$
conditioned on $\xS_n\in\bbE^n$.
Furthermore, since the Vervaat transformation (see \refS{Strees})
is an $n$-to-$1$ map of the set of bridges $\bbB^n$ to $\bbE^n$, which only
permutes the set of increments,
we also have the equality in distribution
\begin{align}\label{bb2}
\bn_{\cT_{\bw,n}}
=
\xpar{n_{\cT_{\bw,n}}(i): i\ge0}
&\eqd
\bigpar{\xpar{|\set{1\le j\le n: \xi_j=i}|:i\ge0}
\mid \xS_n\in\bbB^n}
\notag\\&
=\bigpar{\xpar{|\set{1\le j\le n: \xi_j=i}|:i\ge0}
\mid S_n=n-1}
,\end{align}
where the final equality follows from \eqref{tw1} and \eqref{bb1a}.
We use this to compute factorial moments of 
$n_{\cT_{\bw,n}}(i)$.
Fix $k \in \mathbb{N}_{0}$, and let $q_{0n},\dots,q_{kn}$ be non-negative
integers.
Then \eqref{bb2} shows that
$\prod_{i=0}^k \xpar{n_{\cT_{\bw,n}}(i)}_{q_{in}}$ 
has the same distribution as the number of arrays of distinct indices 
$\xpar{j_{i\ell}: 0\le i\le k, \, 1\le \ell \le q_{in}}$ such that
$\xi_{j_{i\ell}}=i$ for all $i$ and $\ell$, conditioned on 
$S_n=n-1$.
For each choice of indices $\xpar{j_{i\ell}}$,
the probability that $\xi_{j_{i\ell}}=i$ for all $i$ and $\ell$
is $\prod_{i=0}^k w_i^{q_{in}}$, and $S_n=n-1$ if and only if the sum of the
remaining $n-\sum_{i=0}^k q_{in}$ variables $\xi_j$ equals $n-1-\sum_{i=0}^k iq_{in}$.
Hence, \eqref{bb2} yields the factorial moments
\begin{align}\label{bb4}
\E \prod_{i=0}^k \xpar{n_{\cT_{\bw,n}}(i)}_{q_{in}}
= (n)_{\sum_{i=0}^k q_{in}}\prod_{i=0}^k w_i^{q_{in}}\cdot
\frac{\P\bigpar{S_{n-\sum_{i=0}^k q_{in}}=n-1-\sum_{i=0}^k iq_{in}}}
{\P\xpar{S_{n}=n-1}}.
\end{align}

Note that we may ignore indices $i$ with $w_i=0$ in \eqref{fi11}, since 
then $n_{\cT_{\bw,n}}(i)=0=\theta_i(\bw)$, while \eqref{fi12}--\eqref{fi13} yield
$\dgamma(i,j)=0$ for every $0 \le j \le k$.
Hence, in the sequel we 
assume $q_{in}=0$ when $w_i=0$, which means that we really
consider only $0 \leq i\le k$ with $w_i>0$.
Also, assume that $q_{in}\le C\sqrt n$ for all $0 \leq i\le k$ and some fixed constant
$C >0$. 
We estimate below \eqref{bb4} as \ntoo;
the estimates  will be uniform for all such $(q_{in})_{i=0}^k$.

Let $\mu_{in}:=nw_i$ and $\gs_{in}:=\sqrt n$, for $0 \leq i \leq k$. Then \eqref{bb4} and \refL{L0}
yield
\begin{align}\label{bb5}
\E \prod_{i=0}^k \xpar{n_{\cT_{\bw,n}}(i)}_{q_{in}}
= \prod_{i=0}^k\mu_{in}^{q_{in}}
\cdot\exp\lrpar{-\frac1{2n}\Bigpar{\sum_{i=0}^k q_{in}}^2+O\xpar{n\qqw}}
\frac{\P\bigpar{S_{n-\sum_{i=0}^k q_{in}}=n-1-\sum_{i=0}^k iq_{in}}}
{\P\xpar{S_{n}=n-1}}.
\end{align}
Hence we see that what remains to apply \refT{TGW} is a suitable local limit
theorem for the sums $(S_n, n \geq 0)$. We consider the cases \ref{case1GW}
and \ref{case2GW} separately.

\pfitemref{case1GW}
In this case, the distribution $\bw=\gth(\bw)$ has mean $1$ and finite
variance $\gss_\bw>0$.
Let $h\ge1$ be the span of this distribution, i.e., the largest integer such
that $\xi_1$ a.s.\ is a multiple of $h$. Then $Z_n(\bw)>0$ only if $n\equiv
1\pmod{h}$, so we consider only such $n$;
moreover, by assumption
$q_{in}>0$ only if $w_i>0$ and thus $i\equiv0\pmod h$. 
Hence, the standard local
central limit theorem, see e.g.\ \cite[Theorem 7.1]{Petrov},
yields, recalling $q_{in}=O(\sqrt n)$,
\begin{align}
 \P\Bigpar{S_{n-\sum_{i=0}^k q_{in}}=n-1-\sum_{i=0}^k iq_{in}}
&=\frac{h}{\sqrt{2\pi\gss_\bw n}}
\exp\lrpar{-\frac{\lrpar{\sum_{i=0}^k iq_{in}-\sum_{i=0}^k q_{in}}^2}{2\gss_\bw n}+o(1)},
\\
 \P\Bigpar{S_{n}=n-1}
&=\frac{h}{\sqrt{2\pi\gss_\bw n}}
\exp\lrpar{o(1)}.
\end{align}
Hence,
\eqref{bb5} yields, recalling \eqref{fi12}--\eqref{fi13} and our assumption
$\gth(\bw)=\bw$, 
\begin{align}\label{bb8}
\E \prod_{i=0}^k \xpar{n_{\cT_{\bw,n}}(i)}_{q_{in}}
&= \prod_{i=0}^k\mu_{in}^{q_{in}}
\cdot\exp\lrpar{
-\frac{\lrpar{\sum_{i=0}^k q_{in}}^2+\lrpar{\sum_{i=0}^k (i-1)q_{in}}^2/\gss_\bw}{2 n}
+o(1)}
\notag\\&
= \prod_{i=0}^k\mu_{in}^{q_{in}}
\cdot\exp
\lrpar{\frac12\sum_{i,j=0}^k\frac{n\dgamma(i,j)-\gd_{ij}\mu_{in}}{\mu_{in}\mu_{jn}}q_{in}q_{jn}
+o(1)}.
\end{align}
Hence, \refT{TheoDegreeS} follows 
by \refT{TGW} (with $\gs_{in}:=\sqrt n$) in case \ref{case1GW}.

\pfitemref{case2GW}
This is similar. By assumption, there exist sequences of constants $a_n>0$
and $b_n$ such that $(S_n-b_n)/a_n$ converges in distribution to some stable
random variable $Y$ of index $\alpha \in (1,2]$.
Let again $h$ be the span of the distribution $\bw$, and let $g_Y(x)$ be the
density function of the stable limit $Y$.
Then, a local limit theorem holds, see e.g.\ 
\cite[\S\ 50]{GneKol}.
Since we assume that the variance of $\xi_1$ is
infinite, we have $a_n\gg n\qq$, see \cite[XVII.(5.23)]{FellerII}, and thus
$q_{in}=o(a_n)$, and then the local limit theorem yields, simply,
\begin{align}
 \P\Bigpar{S_{n-\sum_{i=0}^k q_{in}}=n-1-\sum_{i=0}^k iq_{in}}
&=\frac{h}{a_n} g_Y(0)
\exp\lrpar{o(1)},
\\
 \P\Bigpar{S_{n}=n-1}
&=\frac{h}{a_n} g_Y(0)
\exp\lrpar{o(1)},
\end{align}
and hence, since $g_Y(0)>0$ 
(see e.g.\ \cite[Remark 4 after Theorem 2.2.3, pp.\ 79--80]{Zolotarev1986}),
\begin{align}\label{bb11}
\frac{\P\bigpar{S_{n-\sum_{i=0}^k q_{in}}=n-1-\sum_{i=0}^k iq_{in}}}
{\P\xpar{S_{n}=n-1}}
\to1,
\end{align}
as \ntoo.
It follows from \eqref{bb5} that \eqref{bb8} holds in this case too, now
with $\gss_\bw=\infty$, and the proof in case \ref{case2GW}
is completed as above by \refT{TGW}.
\end{proof}

\begin{remark}\label{RB1}
  We see from the proof above that \refT{TheoDegreeS} holds for any weight
  sequence $\bw$ such that the probability distribution $\gth(\bw)$
  satisfies a nice local limit theorem for the sums $(S_n, n \geq 0)$.
We do not know whether that holds in full generality, but we note that any
extension of the local limit theorems used above gives a partial answer to 
\refP{ProbPaul}.
However, it is conceivable that there are weight sequences for which such a
local limit theorem fails, but nevertheless \refT{TheoDegreeS} holds.
\end{remark}
\end{appendices}


\providecommand{\bysame}{\leavevmode\hbox to3em{\hrulefill}\thinspace}
\providecommand{\MR}{\relax\ifhmode\unskip\space\fi MR }
\providecommand{\MRhref}[2]{%
  \href{http://www.ams.org/mathscinet-getitem?mr=#1}{#2}
}
\providecommand{\href}[2]{#2}

\end{document}